\documentclass[leqn,12pt]{article}
\usepackage{graphicx}
\usepackage{epsfig}
\usepackage[utf8]{inputenc}
\usepackage[T1]{fontenc}
\usepackage{amsfonts,amsmath}
\usepackage{latexsym}
\usepackage{float} 
\usepackage{multicol} 
\usepackage{amssymb}
	\usepackage{theorem}
	\DeclareMathOperator*{\argmax}{arg\,max} 
	\def\esssup{\mathop{\rm ess\, sup}}
\usepackage{subcaption}
\usepackage{cases}
\usepackage{color}
\def\lbr{[\![}	
\def\rbr{]\!]}

\def \Lim{\displaystyle\lim}
\def \Inf{\displaystyle\inf}
\def \Sup{\displaystyle\sup}

\def \Max{\displaystyle\max}
\newlength{\oldparindent}
\newcommand{\pf}{\hfill$\bullet$}
\usepackage{hyperref}
\usepackage{url}
\title{ Optimal stopping contract for Public
Private Partnerships under moral hazard}
\author{Ishak Hajjej\thanks{\small  CREST $\&$  ENSIIE,  LaMME, France, email \href{mailto: ishak.hajjej@ensiie.fr}{ ishak.hajjej@ensae.fr}  }, Caroline Hillairet\thanks{\small CREST, ENSAE Paris, France, email \href{mailto: Caroline.Hillairet@ensae.fr}{ Caroline.Hillairet@ensae.fr}},
Mohamed Mnif\thanks{\small  ENIT, LAMSIN, University of Tunis El Manar, Tunis, Tunisia, email \href{mailto:mohamed.mnif@enit.rnu.tn}{mohamed.mnif@enit.rnu.tn}   } }
\newtheorem{definition}{Definition }[section]
\newtheorem{proposition}[definition]
{Proposition }
\newtheorem{lemme}[definition]%
{Lemma }
\newtheorem{theoreme}[definition]%
{Theorem }
\newtheorem{corollaire}[definition]%
{Corollary }
\newtheorem{remarque}[definition]%
{Remark }
\newtheorem{hypothese}[definition]%
{Assumption}

{Example}
\newcommand{\esp}{\mathbb{E}} 

\newcommand{\un}{{\mathbf{1}}} 
\newcommand{\proba}{\mathbb{P}}
\newcommand{\pr}{\mathbb{P}}

\newcommand{\ff}{\mathbb{F}}

\newcommand{\R}{\mathbb{R}}
\newcommand{\N}{\mathbb{N}}

\newcommand{\F}{\mathcal{F}}

\newcommand{\demi}{\frac{1}{2}}

\numberwithin{equation}{section}
\begin{document}
\maketitle
\section*{Abstract}
This paper studies optimal Public Private  Partnerships contracts  between a public entity and a consortium, in continuous-time and with a  continuous payment,  {and} the possibility for the public to stop the contract. The public ("she") pays a continuous rent to the consortium ("he"), while the latter gives a best response characterized by his effort. This effort impacts the drift of the social welfare,  until a terminal date decided by the public when she stops the contract and gives compensation to the consortium. Usually,  the public cannot observe the effort done by the consortium, leading to a  principal agent's problem with moral hazard.  {Therefore this paper formalizes such PPP contracts into a contract theory problem. Due to the long-term characteristic of PPP contracts, the public should incentivize the consortium to provide effort  not only through  the terminal payment but also  through the rent paid  until the end of the contract.}
We solve this  optimal stochastic control with  optimal stopping problem in this context of moral hazard. 
 The public value function is characterized by the  solution of an associated  Hamilton Jacobi Bellman Variational Inequality. The public value function,  the  optimal effort and rent processes  are computed  numerically by using the Howard algorithm. In particular,  the impact of the social welfare's volatility on the optimal contract is studied.

{\it Keywords}: Moral Hazard, Public Private Partnership, stochastic control, optimal stopping,  Hamilton Jacobi Bellman Variational Inequality, Howard algorithm.

{\it MSC Classification} : 60G40, 91B40, 91B70, 93E40.

{\it Funding}: This research is supported by a grant of the French National Research Agency (ANR), ''Investissements d'Avenir'' (LabEx Ecodec/ANR-11-LABX-0047).

{\it Acknowledgments}: The authors thank Nizar Touzi, Said Hama\`ene ,  Monique Pontier,  Alexandre Popier  and anonymous referees for fruitful discussions and constructive suggestions.

\section{Introduction} 
Public-private partnership (PPP) is defined as a long-term contract between a private party and a public entity, for the construction and/or the management of an asset or public service.
 Typically, the consortium is making an effort to improve the social value of the project in exchange to  a rent payed by the public.  
 { One main advantage of PPP contracts for the public entity  is to outsource the investment (and thus the debt), see Espinosa et al. \cite{espinosa2016reducing}. 
  Another motivation is to transfer part of the risk to the consortium. Nevertheless, this risk transfer is not  always efficient as pointed out in  Hillairet et al. \cite{hillairet2012modelization} which shows that PPP could be interesting (compared  to standard commissioning of public works) only in some context as high noise, high reference cost, short maturity, and high enough failure penalties.  In their model, as in most economic papers  (see Iossa and Martimort  \cite{iossa2015simple}), the rent is assumed to be a linear rule of the effort of the consortium: although this modelisation leads to tractable computations, it seems very 'ad hoc' and economically questionable. A previous work of Hajjej et al. \cite{hajjej2017optimal}, focusing  on the informational asymmetry issue in PPP contracts, does not assume any a priori form for the rent, and  shows  that the optimal rent is actually not a linear rule. 
Indeed,  one  major concern of this type of contracts is the asymmetry of information between the two parties: public and private partners obviously do not share the same information for negotiation, management and follow-up of the contract.  Auriol and Picard  \cite{auriol2013theory} prove that Build-Operate-Transfer (BOT) contracts (a variant of PPP contracts) may be relevant for the public in case of better information of the private partner, provided a large enough number of concession candidates (although, in France for example,  only three consortium are able to support a PPP contracts).} 
In particular, the public can usually not observe the effort done by the consortium. It is a principal agent's problem with moral hazard. 
Numerous situations in the economic literature lead to principal agent's formulation. For example in 
 Biais et al. \cite{biais2010large}, respectively  in Pag\`es and Possamai \cite{Pages2014AMT}, the unobservable effort of the  agent reduces the intensity of a Poisson process describing  the arrival of large  losses, or respectively  the default time of a  pool of long term loans.

The first paper on principal agent problems in  continuous-time  is the paper of Holmstrom and Milgrom \cite{holmstrom1987aggregation}.  They considered a Brownian setting in which the agent controls the drift of the output process, and receives a lumpsum payment at the end of the contract, that is a finite time horizon. {In their setting, the agent is risk averse with Constant Absolute Risk
Aversion and the principal is risk-neutral.}  Williams \cite{williams2009dynamic} and Cvitanic and Zhang \cite{cvitanic2007optimal} extended those results to more general utility functions. In those situations, the   optimal contract (characterized by the  lumpsum payment) is  a linear function of the output process terminal value.
A general theory, using  coupled systems of Forward Backward Stochastic Differential Equations, is developed in the monograph of Cvitanic and Zhang  \cite{cvitanic2012contract}.  
 Still in a framework of a lumpsum payment on a finite horizon, Cvitanic et al.  \cite{cvitanic2017moral,cvitanic2018dynamic} considered a general formulation in which the agents efforts impact both the drift and  the volatility of the output process, using second-order BSDE in a non Markovian stochastic control setting.

Nevertheless, due to the long maturity of PPP contracts (around 30 to 50 years), it seems unreasonable to propose to the consortium a unique payment at the maturity of the contract. For example,  Hajjej et al. \cite{hajjej2017optimal} derived the optimal perpetual contract (characterized by a rent) using techniques  of stochastic control under partial information. This paper proposes a similar modeling  with a continuous payment in  random horizon, but adding   the possibility of stopping the contract at  a random time, decided by the public. This combines optimal stochastic control and optimal stopping in this context of moral hazard. 
  The seminal paper of Sannikov  \cite{sannikov2008continuous} proposed a tractable model, in  a continuous-time setting and with continuous payment, to study the optimal contract and the  optimal  time of retiring/firing  the agent.   The optimal contract is written  as a function of the agent's continuation value, which appears as the state variable of the problem.  {Recently Possamaï and Touzi \cite {possamai2020there} revisited the seminal paper of Sannikov  \cite{sannikov2008continuous}. {They considered a situation of "Golden Parachute" which corresponds to a situation in which  the agent ceases any effort at some positive stopping time, and receives a payment. They showed  that a Golden Parachute only exists in certain specific circumstances. This contrasts with Sannikov's results,  where the only requirement is a positive agent's marginal cost of effort at zero. They showed that there is no Golden Parachute if this parameter is too large.}}   Anderson et al. \cite{anderson2018agency} studied the optimal replacement time (either for the sake of incentive provision, or for the sake of growth) of managers operating for a long-lived firm.
   D\'ecamps and Villeneuve \cite{decamps2019two} studied the optimal strategic liquidation time, in a framework where the  firm's profitability  is impacted by the unobservable managerial effort of the agent. In this setting  the principal's problem appears to be  a 2-dimensional fully degenerated Markov control problem and  the optimal contract that implements full effort is derived. Both papers \cite {anderson2018agency} and \cite{decamps2019two}  assume the agents and the  principal to be risk neutral.
   { Let us also  highlight  the recent paper of  Lin et al. \cite{lin2020random}  which considers a general formulation of the random horizon principal-agent problem with a
  continuous payment and a lump-sum payment at  a random terminal date. In fact,    \cite{lin2020random}
extends Sannikov's model to the setting where the agent is allowed to control the diffusion of the output process which takes values in $\R^d$.
In our case, 
the social value (which represents the output process) is valued in $\R$ and so controlling the volatility would imply that  the  effort is observable, through the quadratic variation of the social value. This implies, that  in dimension one, controlling the volatility is not a relevant model for moral hazard.
Therefore  \cite{lin2020random} is more general and more technical, using  the second-order Backward SDE theory.  In our case, we use classical tools, whose main advantage is to obtain the structure of the optimal rent and the optimal effort in a feedback form and therefore to allow us  to  compute  numerically the optimal controls.
}

  In this paper, we consider a  contract between a public entity and a consortium, in a continuous-time setting. The consortium is making  effort to improve the social value of the project, driven by a one-dimensional Brownian motion. The effort is not observable by the public, that must choose a continuous rent she will pay to the agent in compensation to his effort.  We assume that the effort only affects the drift and not the volatility of the social value. 
  We also assume that the volatility of the social value is known, contrary to the paper of Mastrolia and Possamaï  \cite{mastrolia2018moral} in which the agent and the principal faced both uncertainty on the volatility of the output. Our aim is to study qualitatively  the impact of the volatility parameter on the optimal contract. 
  Since PPP are contracts covering decades, our model  tackles the possibility for the public to stop the contract at a random date. 
  The public pays a rent to the consortium, while the latter gives a best response characterized by his effort until a terminal date decided by the public when she stops the contract and gives compensation to the consortium.  We assume that the  consortium  will accept the contract only if his expected payoff exceeds his reservation value $\underline{x}.$ As in Sannikov \cite{sannikov2008continuous}, we assume  that the agent is risk averse and that the principal is risk neutral. 
  We consider a Stackelberg leadership model between the public and the consortium, that can be solved in two steps. First, given a fixed contract, the public computes the best effort of the consortium. Then, the public solves her problem by taking into account the best effort of the consortium  and computes the associated optimal contract.\\
  {Our approach is inspired by the seminal paper of Sannikov  \cite{sannikov2008continuous}, in a more general setting. Our main contribution is to establish a one-to-one correspondence between the continuation value of the consortium and the contract (rent plus terminal) payments,  using    Backward Stochastic differential equations  (BSDE) with terminal  time, stochastic control  and optimal stopping technics.}
As it is standard in the literature, we use the weak approach, that is  the agent changes the distribution of the  social value of the project, by making the probability measure  depend on agent's effort.  We derive the Hamilton Jacobi Bellman variational inequality characterizing the public value function. 
Finally we   provide numerical solutions using Howard algorithm.   Our contribution is twofold: first we provide rigorous results, combining optimal stochastic control and optimal stopping in this context of moral hazard, which is a challenging task especially in this literature.  Then we detail the procedure to compute the numerical solutions and we provide  a numerical analysis of the characteristics of the optimal contracts (effort, rent, value function) as well as  the sensitivity with respect to the diffusion coefficient $\sigma$ of the social value of the project.\\
The outline of the paper is as follows. In Section \ref{section 2}, we formulate the problem, using the weak approach and we describe the public and the consortium problems. In Section \ref{section3}, we determine the incentive compatible contract and we provide the dynamics of the consortium objective function, using the BSDE technique. In Section \ref{section4}, we derive the  Hamilton Jacobi Bellman Variational Inequality associated  to the public value function and we provide a verification theorem.  Section \ref{section5}  is devoted mainly to the numerical study of the  Hamilton Jacobi Bellman Variational Inequality based on the Howard algorithm,    { and the computations of the optimal controls and the value functions}. Technical results on BSDE with random horizon are postponed in the Appendix.

\section{The Public Private Partnership model under moral hazard}\label{section 2}
In this section, we work under the  weak formulation as it is done  usually in the principal-agent literature.
 Let $W$ be a standard one dimensional  Brownian motion defined on a probability space $(\Omega,\pr)$, and $\mathbb{F}=(\mathcal{F}_t)_{t\geq0}$  is  the  filtration generated by $W$, satisfying the usual conditions of right-continuity and completeness.\\
We introduce the social value of the project that is observed by the public
\begin{equation}\label{social}{
X_t :=X_0+\sigma W_t},~ t\geq0,~\pr~a.s
\end{equation}
 where \vspace{-0.5cm}
 \begin{itemize}
 \item $X_0>0$ is the initial value of the project.
 \item { $\sigma>0$} is the volatility of the operational cost of  the infrastructure maintenance, that is assumed to be constant.\footnote{This could be generalized to a regular $\mathbb F$-progressive map $\sigma(X_t)$, this will be further discussed  in Section \ref{Rksigmamap}.}
  \end{itemize}

\begin{remarque} 
We could take $\sigma$ equal to $1$ in (\ref{social}). If $\sigma \equiv 1$,  then $X$ is the canonical process defined on $\Omega$ the set of continuous paths starting from $X_0$ and $\pr$ is the Wiener measure on $\Omega.$ In this paper, $\sigma \not\equiv 1$  and  we  study numerically the sensitivity of the value function, the optimal effort and the optimal rent to the volatility parameter $\sigma.$
\end{remarque}
In the weak formulation,  the agent changes the distribution of the  process $X$, by making the underlying probability measure $\pr^A$ depend on agent's effort $A$. {We define the process $\gamma^A=(\gamma_t^A)_{t\ge 0}$ by
 $$\gamma_t^A :=\exp\left[\int_0^t\frac{\varphi(A_s)}{\sigma}dW_s-\demi\int_0^t\left(\frac{\varphi(A_s)}{\sigma}\right)^2ds\right]= \mathcal{E}\left[\int_0^t\frac{\varphi(A_s)}{\sigma}dW_s\right] ~ t\geq0,~\pr~a.s$$ 
where $\varphi$ is  specified hereafter.
{We denote by $ {\cal{T}}$ the set of all $\ff$-stopping times} {and for $t\geq0,~\mbox{we define the set }\mathcal{T}_t:=\{\tau\in\mathcal{T} \mbox{ such that } \tau\geq t\}$}. 
 We fix $\hat p\in (2,\infty)$  and we consider
\begin{eqnarray}\label{setA}
\mathcal{A}^{\hat p }&:=&\{(A_s)_{s\geq0}~\ff\mbox{-progressively measurable process,}~ A_s\geq0~ds\otimes d\pr~ a.e.~\nonumber \\
& & \quad  \mbox{such that}~ \esssup_{\tau\in\mathcal{T}_t}\esp[(\frac{\gamma^A_\tau}{\gamma^A_t})^{\hat p}|\mathcal{F}_t]<\infty,~\forall~t\geq0\}.
\end{eqnarray}
\noindent The probability measure $\pr^A$ is defined by {$\frac{d\pr^A}{d\pr}|_{\F_t}=\gamma^A_t $~for all $t\geq0$}. Then, $\pr^A$ and $\pr$ are  equivalent. By Girsanov's theorem,
 the process $(W_t^A)_{t\geq 0}$ defined by 
  $$W_t^A:= W_t-\int_0^t\dfrac{\varphi(A_s)}{\sigma}ds,~\mbox{for}~ t \geq 0$$
 is a $\pr^A$-Brownian motion and the social value of the project is given by:

 \begin{equation}\label{eq1}
X_t = X_0+\int_0^t\varphi(A_s)ds+\sigma W_t^A,~ t\geq0,~\pr~a.s.
 \end{equation}
The public observes the social value $X$ of the project, but she does not observe directly the effort of the consortium: this  is a situation of moral hazard. 
She chooses the rent she will pay to the consortium to compensate him for his efforts and the operational costs that he supports.
The public could end the contract at the date $\tau$, where $\tau $ is a stopping time in ${\cal{T}}.$ \\
A contract is a triplet $\Gamma=((R_t)_t,\tau,\xi)$ where $R$, which represents the rent paid by the public, is a non-negative ${\ff}${-progressively measurable process}, $\tau\in{\cal{T}}$, and
$\xi$ is  a non negative ${\F}_\tau$-mesurable {random variable} which represents the cost of stopping the contract.
 \begin{remarque}\label{remarque1} Contrary to a strong  formulation (cf. Hajjej et al. \cite{hajjej2017optimal}), the filtrations $\ff^X$ and $\ff$ coincide in the weak formulation, where $\ff^X$ is the filtration generated by the social value process $X$ and $\ff$ is the Brownian filtration generated by the standard Brownian motion $W$.
\end{remarque}
We now define the respective optimization problems for the consortium and the public.
Let us first define the functions involved in the formulation of the optimization problems.
\begin{hypothese}\label{hyp}
\begin{itemize}
\item[\pf]$\varphi$ is the function that models the marginal impact of the consortium's efforts on the social value, $\varphi:[0,\infty) \rightarrow[0,\infty)$ is {$C^2$} { strictly} concave, bounded, increasing, {$\varphi(0)=0$} and $\varphi'(0)>0$. We denote by $\|\varphi \|_{\infty}:=\sup_{a \geq 0 }\varphi(a)$.
\item [\pf] The utility function of the consortium {$U :[0,\infty) \rightarrow[0,\infty)$} is {$C^2$} strictly concave increasing and satisfying $U(0)=0$ and  Inada's conditions $U'(\infty)= 0,~U'(0)=\infty$.
\item [\pf] $h$ is the cost of the effort for the consortium; $h:[0,\infty) \rightarrow[0,\infty)$ is {$C^2$}, {strictly} convex {increasing}, $h(0)=0$.
\end{itemize}
\end{hypothese}
{In this paper, the time preference parameter $\lambda$ of the consortium is not necessarily equal
to $\delta$ the one of the public, in general it is usual to assume that the
consortium is more impatient than the public ($\lambda\geq\delta$).}\\

\noindent Given
 a contract $\Gamma=((R_t)_{t},\tau,\xi)$ offered by the principal, the consortium gives a best response in terms of an effort
process $A$: this is a Stackelberg leadership model. The consortium accepts the contract only if his expected payoff exceeds his reservation value $\underline{x}.$ 
\begin{enumerate}
\item {{Agent's best response}}
\begin{equation}\label{bestresp}
A^* \in \argmax_{A\in \mathcal{A}^C_{\rho-{2 \lambda}}} \esp^A\left[\int_0^\tau e^{-{{\lambda}} s}(U(R_s)-h(A_s))ds+e^{-{{\lambda}} \tau} U(\xi){\un_{\{\tau<\infty\}}}\right]
\end{equation}
where $\esp^A$ is the expectation under $\pr^A$, and  for some $\rho>0,$ 
\begin{eqnarray*}\label{admissible}
 \mathcal{A}^C_{\rho-{ 2 \lambda}}&:= &\Big\{ (A_s)_{s\geq 0 }\in\mathcal{A}^{\hat p},~\mbox{such that}~\esp ^\pr \left[{\int_0^{\infty}} e^{(\rho-{ 2 \lambda}) s}|h(A_s)|^2 ds \right]<\infty\\
& &\quad \mbox{and}~\esp ^\pr \left[{\int_0^{\infty} }e^{(\rho-{  2 \lambda}) s}|\varphi(A_s)|^{2}ds \right]<\infty\Big\},
\end{eqnarray*}
\noindent The {\it{objective function}} starting at time $t$ for the consortium is $\pr^A$-a.s.
\begin{equation*}
J_t^C(\Gamma,A) :=\esp^A\left[\int_t^\tau e^{-{{{\lambda}}(s-t)}}(U(R_s)-h(A_s))ds+e^{-{{\lambda}}(\tau-t)}{  U(\xi)}{\un_{\{\tau<\infty\}}}|\F_t\right].
\end{equation*}  
As the process $(J_t^C(\Gamma,A))_{0\leq t\leq \tau}$, is continuous and  $\pr^A\sim\pr$, we have
\begin{equation}\label{Objcons}
J_t^C(\Gamma,A) =\esp^A\left[\int_t^\tau e^{-{{{\lambda}}(s-t)}}(U(R_s)-h(A_s))ds+e^{-{{\lambda}}(\tau-t)}{  U(\xi)}{\un_{\{\tau<\infty\}}}|\F_t\right],~\forall t\in\lbr 0,\tau \rbr,~\pr ~a.s.
\end{equation} 
in which $\lbr .,.\rbr$ denotes a stochastic interval.
\item Given the best response of the agent, the public entity (the principal) aims  to maximize the social value of the project minus the rent. Since  $\esp^{ { A}} \left[ \int_0^{\tau}e^{-\delta s} dX_s  \right]=\esp^{ { A}} \left[ \int_0^\tau e^{-\delta s}\varphi({A_s})ds\right]$ by Doob's optional sampling theorem,
 the principal problem is formulated by

\begin{equation}\label{Principal}
\Sup_{\Gamma\in \mathcal{A}^P_{\rho-{  2 \lambda}}}\Sup_{P^{A^*}\in {  \mathcal{P}^{\Gamma}}} \esp^{ { A^*}}\left[\int_0^\tau e^{-\delta s}(\varphi({A_s^*})-R_s)ds-e^{-\delta \tau} \xi{\un_{\{\tau<\infty\}}}\right]
\end{equation}
subject to the reservation constraint
$$\esp^{ A^*}\left[\int_0^\tau e^{-{{{\lambda}} s}}(U(R_s)-h({ A_s^*}))ds+e^{-{{{\lambda}}} \tau}  {  U(\xi)}{\un_{\{\tau<\infty\}}}\right]\geq {\underline{x}}$$
with \vspace{-0.2cm}
\begin{eqnarray}\nonumber\label{AP}
\mathcal{A}^P_{{\rho-{  2 \lambda}}}&:= &\Big{\lbrace} ((R_s)_{s\geq 0 } ,\tau,\xi)~\mbox{ : }~ R~\mbox{ is } \mathbb{F}\mbox{-progressively measurable process,}
\\
& &
~ \quad R_s\geq 0~ds\otimes d\pr~ a.e.\mbox{ and }~\esp^\pr \left[{\int_0^\infty} e^{(\rho-{  2 \lambda}) s}(U(R_s)^2  \vee R_s^2)ds \right]<\infty,  \\
& & \quad \tau\in\mathcal{T}, ~\xi\geq 0~\F_\tau\mbox{-measurable s.t. }~\esp^\pr \left[e^{(\rho-{  2 \lambda}) \tau}( {  U(\xi)^{2} \vee } \xi^2)\un_{\{\tau<\infty\}} \right]<\infty
\Big{\rbrace},\nonumber
\end{eqnarray}
where  $x \vee y:=\max (x,y)$ and 
\begin{eqnarray*}
{  \mathcal{P}^{\Gamma}}:=\left\{\pr^{A^*}\sim\pr,A^*\in \argmax_{A\in \mathcal{A}^C_{{\rho-{  2 \lambda}}}} \esp^A\left[\int_0^\tau e^{-{{\lambda}} s}(U(R_s)-h(A_s))ds+e^{-{{\lambda}} \tau} {  U(\xi)} {\un_{\{\tau<\infty\}}}\right]\right\}.
\end{eqnarray*}
The {\it{objective function}} { starting at} time $t$ for the public is $\pr^{A^*}$-a.s.
\begin{equation*}
J_t^P( \Gamma,{ A^*}) :=\esp^ {A^*}\left[\int_t^\tau e^{-\delta(s-t)}(\varphi({A_ s^*})-R_s)ds-e^{-\delta (\tau-t)}\xi{\un_{\{\tau<\infty\}}}| \F_t\right].
\end{equation*}
Using the same arguments as in (\ref{Objcons}), we have
\begin{equation}
J_t^P( \Gamma,{ A^*}) =\esp^ {A^*}\left[\int_t^\tau e^{-\delta(s-t)}(\varphi({A_ s^*})-R_s)ds-e^{-\delta (\tau-t)}\xi{\un_{\{\tau<\infty\}}}| \F_t\right],~\forall t\in\lbr 0,\tau \rbr,~\pr ~a.s.
\end{equation}
\end{enumerate}

\section{Incentive compatible contracts}\label{section3}
The aim of this section  is to determine the incentive compatible contracts and to provide the dynamics  of the consortium objective function $J^C$.
To achieve this, one  first needs to prove an existence and uniqueness result for a certain type of BSDE with random horizon.}\\
As we will see later, the objective function for the agent is related to the solution of the following BSDEs with a random time horizon  $\tau$
\begin{equation}\label{BSDE02}
Y_t={\zeta}{\un_{\{\tau<\infty\}}}+\int_{t}^\tau g(s,\omega,Z_s)ds-\int_t^\tau Z_s dW_s,
\end{equation}
where the generator $g:\R_+\times\Omega\times\R\rightarrow\R$ does not depend on $Y$.\\
BSDEs with random horizon have been studied by some authors. 
 {Chen \cite{chen1998existence}  considered a random horizon which could be infinite and assumed that the constant of Lipschitz  {of the generator} is time-dependent and  square integrable  on $[0,\infty]$:  this  assumption  is not satisfied  in our case, since  the Lipschitz coefficient {$C_g$} is constant. Darling and Pardoux \cite{darling1997backwards} studied a BSDE with {finite} random horizon.  {In our case, we cover the finite and the infinite random horizon.}  We give  in the Appendix  the proof of the existence and uniqueness of a solution to the BSDE (\ref{BSDE02}), under the following assumptions
\begin{itemize}
\item [{ (H1($\kappa$))}] The generator  $g:[0,\infty)\times\Omega\times\R\longrightarrow \R$ satisfies, for each $z\in\R$, $g(.,.,z)$ is a progressively measurable process and for some {  $\kappa \in \R$}, we have
 $$\esp\left[e^{{  \kappa}\tau}|\zeta|^2{\un_{\{\tau<\infty\}}}+\int_0^\tau e^{{  \kappa} s}|g(s,\omega,0)|^2ds\right]<\infty.$$
\item [(H2)] $g$  is Lipschitz with respect to $z$, i.e. there exists a constant $C_g>0$ such that 
$$|g(s,\omega,z_1)-g(s,\omega,z_2)|\leq C_g|z_1-z_2|~ds\otimes d\pr~{ a.e.}$$
\end{itemize}
We introduce the following spaces for a fixed stopping time $\tau\in \mathcal{T}$ and for {some ${  \kappa \in \R}$:}
\begin{eqnarray*}
\mathcal{S}^2_{  \kappa}(\tau):&=&\{  {Y}~ \R\mbox{-valued,}~\ff\mbox{-progressively measurable continuous process such that}\\
& &
~~~~~~~~~~~~||{Y}||_{{\cal S}^2_{  \kappa}(\tau)} := \left(\esp^\pr \left[\Sup_{ 0\leq s\leq\tau} e^{{  \kappa} s}|{Y}_s|^2\right]\right)^{\frac{1}{2}}<\infty\},\\
{\cal H}^2_{  \kappa}(\tau):&=&\{{Z}~ \R\mbox{-valued,}~\ff\mbox{-progressively measurable process such that}\\
& &
~~~~~~~~~~~~||{Z}||_{{\cal H}^2_{  \kappa}(\tau)} := \left(\esp^\pr\left[ \int_0^{\tau} e^{{  \kappa} s}|{Z}_s|^2ds\right]\right)^{\frac{1}{2}}<\infty\},\\
{L}^2_{  \kappa}(\F_\tau):&=&\{ \zeta~ \R \mbox{-valued,}~ \mathcal{F}_\tau\mbox{-measurable random variable such that}~\esp \left[e^{{  \kappa}
\tau}|{\zeta}|^2{\un_{\{\tau<\infty\}}} \right]<\infty\}.
\end{eqnarray*}
{ The main result of Section \ref{section3} is the following theorem.
\begin{theoreme}
Suppose Assumption \ref{hyp} and  $\rho>\frac{\|\varphi \|_{\infty}^2}{\sigma^2}$, then\\
(i) There exists $Z\in {\cal H}^2_\rho(\tau)$ and a function $A^*:\R\longrightarrow \R$ such that the optimal effort satisfies
\begin{equation}
A^*_t=A^*(Z_t), ~ \forall t\in\lbr 0,\tau \lbr,~\pr ~a.s.
\end{equation}
(ii) For any admissible contract $\Gamma\in\mathcal{A}^P_{ \rho-2\lambda}$,
the consortium value function $(J_t^C(\Gamma,A^*(Z)))_t$ solves the BSDE with random horizon:
\begin{eqnarray}\label{BSDETh}
  dJ_t^C(\Gamma,A^*(Z))&=&-\Big( -{{\lambda}} J_t^C(\Gamma,A^*(Z))+U(R_t)+\psi(A^*(Z_t),Z_t)\Big) dt+Z_tdW_t,\\
  J_\tau^C(\Gamma,A^*(Z))&=& { U(\xi)} {\un_{\{\tau<\infty\}}} \nonumber
\end{eqnarray}
where 
\begin{equation}\label{psii}
\psi(a,z):=-h(a)+{z}\frac{\varphi(a)}{\sigma}.
\end{equation}
\end{theoreme}
The proof is split in several lemmas and propositions, with  first  the existence and unicity of BSDE \eqref{BSDETh}. By a change of variable, this is related to the}
 existence and uniqueness of a solution to the BSDE (\ref{BSDE02}) which is given in the following proposition.
\begin{proposition}\label{EDSR}
{  Let $\tau$ be a stopping time in $\mathcal{T}$ and  $\rho>C^2_g$. If  ${\zeta\in{L}^2_\rho(\F_\tau)},~g$ satisfies {Assumptions} (H1($\rho$)) and (H2), }
	there exists a unique solution $(Y,Z)\in\mathcal{S}^2_\rho(\tau)\times{\cal H}^2_\rho(\tau)$ to the BSDE (\ref{BSDE02}).
\end{proposition}}
\noindent The proof is postponed in the Appendix. { In our setting, the constant $C^2_g$ will be  $\frac{\|\varphi \|_{\infty}^2}{\sigma^2}$}.\\
Thanks to the boundedness of the function $\varphi$, the generator of the BSDE that the consortium value function must satisfy is Lipschitz with respect to the variable $z$. To the best of our knowledge, such condition is crucial to prove the existence of a unique pair $(Y,Z)$ satisfying a BSDE with  random horizon (see Darling and Pardoux \cite{darling1997backwards}).  \\
\noindent  Proposition \ref{EDSR} is used to determine the incentive compatible contract and to provide the dynamics {of the consortium objective function}. This is useful in what follows  in order  to reformulate the optimization problems in terms of the consortium objective function $J^C$.

\begin{lemme}\label{lemmeincentive} Suppose Assumption \ref{hyp} and  $\rho>\frac{\|\varphi \|_{\infty}^2}{\sigma^2}$.
For any admissible contract $\Gamma\in\mathcal{A}^P_{  \rho-2\lambda}$ and for any $A\in\mathcal{A}^C_{  \rho-2\lambda}$,  there exists  $Z^A\in {\mathcal{H}^2_{\rho-{  2\lambda}}(\tau)}$ such that  the dynamics of the consortium objective function evolves according to the BSDE with random terminal condition
\begin{equation}\label{Ii01}
dJ_t^C(\Gamma,A)=-\Big( -{{\lambda}} J_t^C(\Gamma,A)+U(R_t)+\psi(A_t,Z_t ^A)\Big) dt+Z_t^AdW_t,~J_\tau^C(\Gamma,A)= {  U(\xi)} {\un_{\{\tau<\infty\}}}
\end{equation}
where $\psi$ is defined in \eqref{psii}.
\end{lemme}
\begin{proof}
For any admissible contract $\Gamma\in \mathcal{A}^{P}_{  \rho-2\lambda},$ for any $A\in \mathcal{A}^C_{  \rho-2\lambda}$ and for any $ t\in\lbr 0,\tau \lbr$, we define the integrable  process
\begin{equation*}
\begin{array}{rcl}
M_t^C(\Gamma,A) :&=& ~e^{-\lambda t}J_t^C(\Gamma,A)+\int_0^t e^{-\lambda s} (U(R_s)-h(A_s))ds\\
&=&\esp^{\mathbb{P}^{A}}\left[\int_0^\tau e^{-\lambda s}(U(R_s)-h(A_s))ds+e^{-\lambda \tau} U(\xi){\un_{\{\tau<\infty\}}}|\F_t\right]\\
&=&\dfrac{1}{\gamma_t^A}\esp^\proba\left[\gamma_\tau^A\int_0^\tau e^{-\lambda s}(U(R_s)-h(A_s))ds+\gamma_\tau^Ae^{-\lambda \tau} U(\xi){\un_{\{\tau<\infty\}}}|\F_t\right],
\end{array}
\end{equation*}
where the last equality is obtained using Bayes formula.\\
{As $\big( \esp ^{\pr} \left[ \gamma_{\tau}^A(\int_0^\tau e^{-\lambda s}(U(R_s)-h(A_s))ds+e^{-\lambda \tau} U(\xi){\un_{\{\tau<\infty\}}})|\F_t   \right]    \big)$ is a ($\pr$, $\mathbb{F})$-{local martingale}, by the martingale representation theorem, there exists an
 $\ff$-progressively measurable process $\chi$ such that:}
 \footnote{It is important to work under  the probability measure $\pr$ and not $\pr^A$. Indeed, although the inclusion $\ff^A=\sigma(W^A)\subseteq\ff$ holds,  the reverse inclusion is not true in general case (see the Tsirel'son's  example in \cite{yor2012some}).}
 $$M_t^C(\Gamma,A)=\dfrac{1}{\gamma^A_t}\left( M_0^C(\Gamma,A)+\int_0^t \chi_sdW_s\right).$$
Thus $d(\gamma_t^AM_t^C(\Gamma,A))=\chi_tdW_t,$ using It\^o's formula, we obtain
\begin{equation*}
M_t^C(\Gamma,A)=M_0^C(\Gamma,A)+\int_0^t \left(\frac{\chi_s}{\gamma_s^A}{-}M_s^C(\Gamma,A)\frac{\varphi(A_s)}{\sigma} \right)dW_s^A.
\end{equation*}
Then denoting  $Z_s^A =e^{{\lambda} s}(\frac{\chi_s}{\gamma_s^A}{-}M_s^C(\Gamma,A)\frac{\varphi(A_s)}{\sigma})~ ds\otimes d\pr$ a.e,  we deduce 
$$e^{-\lambda t}J_t^C(\Gamma,A)=M_0^C(\Gamma,A)-\int_0^te^{-{\lambda}s}(U(R_s)-h(A_s))ds+\int_0^te^{-{\lambda} s} Z_s^AdW_s^A,$$
and we obtain
\begin{equation*}
dJ_t^C(\Gamma,A)=-\Big( -{\lambda} J_t^C(\Gamma,A)+U(R_t)-h(A_t)\Big)dt+Z_t^AdW_t^A,
\end{equation*}
which implies that under $\pr$
\begin{equation} \label{BSDE0}
dJ_t^C(\Gamma,A)=-\left({-{\lambda}} J_t^C(\Gamma,A)+U(R_t)-h(A_t)+ Z_t^A\frac{\varphi(A_t)}{\sigma}\right)dt+ Z_t^AdW_t,~J_\tau^C(\Gamma,A)={  U(\xi)}{\un_{\{\tau<\infty\}}}.
\end{equation}
The associated Backward Stochastic Differential Equation (BSDE $(A,R,\tau,\xi)$ in short) is given by
\begin{eqnarray}\label{BSDE}
 \left\{
    \begin{array}{ll}
       dY_t &=-\left(-{\lambda}Y_t+U(R_t)+\psi(A_t,Z_t^A)\right)dt+ Z_t^AdW_t, \\
        Y_\tau &={  U(\xi)}{\un_{\{\tau<\infty\}}},
    \end{array}
\right. 
\end{eqnarray}
where {$\psi$ is defined by (\ref{psii}).}
Considering the discounted quantities 
 $$(\tilde{Y}_t,\tilde{Z}_t^A)=(e^{-\lambda t}Y_t,e^{-\lambda t} Z_t^A),~ t\in\lbr 0,\tau \lbr,~\pr ~a.s,$$
$(\tilde{Y},\tilde{Z}^A)$ satisfies the BSDE
\begin{eqnarray}\label{BSDE2}
 \left\{
    \begin{array}{ll}
       d\tilde{Y}_t &=-\left(\tilde{U}(R_t)+\tilde{\psi}(A_t,\tilde{Z}_t^A)\right)dt+\tilde{Z}^A_tdW_t, \\
        \tilde{Y}_\tau &={  \tilde U(\xi)}{\un_{\{\tau<\infty\}}},
    \end{array}
\right. 
\end{eqnarray}
where \vspace{-0.5cm}
\begin{eqnarray}\label{psiU}
 \left\{
    \begin{array}{ll}
           {\tilde{h}(A_t)}&:=e^{-\lambda t}h(A_t),\\
                   \tilde{U}(R_t) &:=e^{-\lambda t}U(R_t),\\ \tilde{U}(\xi) &:=e^{-\lambda t}U(\xi)\\
      \tilde{\psi}(A_t,{\tilde{ Z}}_t^A) &:=-{\tilde{h}(A_t)}+ \tilde{Z}_t^A\frac{\varphi(A_t)}{\sigma}.\\
    \end{array}
\right. 
\end{eqnarray}
The generator of the BSDE (\ref{BSDE2}) depends only on $\tilde{Z}^A$ and is defined by $g(t,\omega,z):=\tilde{U}(R_t)-{\tilde{h}(A_t)}+ z\frac{\varphi(A_t)}{\sigma}.$ {From the definition of the sets $\mathcal{A}^C_{  \rho-2\lambda}$ and $\mathcal{A}^P_{  \rho-2\lambda}$, the  generator $g:\R_+\times\Omega\times\R\rightarrow\R$ satisfies { $(H1(\rho))$} and  $g$ satisfies $(H2)$ since $\varphi$ is bounded. }
Then from Proposition \ref{EDSR}, there exists a unique $(\tilde{Y},\tilde{Z}^A)\in\mathcal{S}^2_\rho(\tau)\times\mathcal{H}^2_\rho(\tau)$ solving the BSDE (\ref{BSDE2}).
Therefore, there exists a unique $(Y, Z^A)$ solving the BSDE (\ref{BSDE}) such that
$$ \left(\esp^\pr \left[\Sup_{0\leq s\leq\tau}e^{(\rho-2\lambda) s} |Y_s|^2\right]\right)^{\frac{1}{2}}<\infty~\mbox{and}~\left(\esp^\pr \left[\int_0^\tau e^{(\rho-2\lambda) s}|Z_s^A|^2ds\right]\right)^{\frac{1}{2}}<\infty,$$
i.e. ${  (Y,Z^A)\in \mathcal{S}^2_{\rho-2\lambda}(\tau)\times\mathcal{H}^2_{\rho-2\lambda}(\tau)}$.\pf \\
\end{proof}
{We are interested in BSDE (\ref{BSDE2}) for which we prove a classic comparison theorem, that is postponed in the Appendix (see Theorem \ref{comparison}). We deduce the following Corollary \ref{comp2}, which is  a direct consequence of the comparison theorem.}
{\begin{corollaire} \label{comp2}
For any $A\in\mathcal{A}^C_{  \rho-2\lambda}$ and for any admissible contract $\Gamma\in\mathcal{A}^P_{  \rho-2\lambda}$,
if there exists $A^*\in\mathcal{A}^C_{  \rho-2\lambda}$ such that
$\psi(A_t,Z_t^A)\leq\psi(A^*_t,Z_t^A)~ \forall t\in \lbr 0,\tau \lbr~ \pr~a.s.,$  then  
$$J_t^C(\Gamma,A)\leq J_t^C(\Gamma,A^*),~\mbox{for all}~t\in \lbr 0,\tau \lbr, ~\pr~a.s.,$$
where $( J_t^C(\Gamma,A^*))_{t\geq0}$ satisfies
\begin{equation}\label{JT2}
\left\{
    \begin{array}{ll}
      dJ_t^C(\Gamma,A^*)&=-\Big(-{{\lambda}} J_t^C(\Gamma,A^*)+U(R_t)+\psi(A^*_t,Z_t ^{A})\Big)dt+Z_t^AdW_t,\\
  J_\tau^C(\Gamma,A^*)&={  U(\xi)}{\un_{\{\tau<\infty\}}} .
    \end{array}
\right. 
\end{equation}
\end{corollaire}}
\noindent The following lemma is technical and will be useful later to parametrize the optimal response  $A^*$ as a deterministic function of $Z^A$ for $ A$ admissible control.
\begin{lemme}\label{lemme01}
Suppose  Assumption \ref{hyp}. Let define $A^*(z):=\arg\Max_{a\geq0}\psi(a,z)$, for $z \in \mathbb R$. \\
If $z>\sigma\frac{h'(0)}{\varphi'(0)},$ then $A^*(z)=(\frac{h'}{\varphi'})^{-1}(\frac{z}{\sigma})$ and if $z\leq\sigma\frac{h'(0)}{\varphi'(0)},$ then $A^*(z)=0.$\\
{Moreover, $A^*$ is  a bijection from $\{0\} \cup(\sigma\frac{h'}{\varphi'}(0),\infty)$ to $[0,\infty).$ }
\end{lemme}

{\begin{proof} 
We fix $z\in\R$. Under Assumption  \ref{hyp}, $~(\frac{h'}{\varphi'})^{'}(a)=\dfrac{h''(a)\varphi'(a)-\varphi''(a)h'(a)}{(\varphi'(a))^2}>0$  for all $a\geq0$. This shows that $\frac{h'}{\varphi'}$ is invertible  from $[0,\infty)$ into $[\frac{h'(0)}{\varphi'(0)},\infty).$\\
{\bf First case}: $z>\sigma\frac{h'(0)}{\varphi'(0)}$. The function $\psi$,   defined in  Equation (\ref{psii}), is strictly concave, the first-order necessary condition of optimality $\frac{\partial\psi}{\partial a}({A^*(z)},z)=0$ is sufficient  and is equivalent to \begin{equation}\label{zz}
z=\sigma\frac{h'(A^*(z))}{\varphi'(A^*(z))}.
\end{equation}
Thus $A^*$ is a bijection from $(\sigma\frac{h'}{\varphi'}(0),\infty)$ to $(0,\infty).$ \\
{\bf Second case}: $0< z\leq\sigma\frac{h'(0)}{\varphi'(0)}$. Then Equation (\ref{zz}) is not well defined. The optimum could not be positive. In the neighborhood of $0$, the function $\psi$ must be decreasing, otherwise $A^*(z)=0$ is not a maximum. The necessary  condition of optimality $\frac{\partial\psi}{\partial a}(0,z)=-h'(0)+z\frac{\varphi'(0)}{\sigma}\leq0$ is satisfied in $0$. As the function $\psi$ is strictly concave, the latter condition is sufficient.
 This shows that $A^*(z)=0.$\\
{\bf Third case}: $z\leq0$. The function $\psi(z,.)$ is decreasing in $a$ and so $A^*(z)=0.$ In particular, $A^*(0)=0,$ and so Lemma \ref{lemme01} is proved. 
\pf\\
\\
\end{proof}}
 { The following proposition gives a representation of the optimal effort. This step is crucial to transform the initial non-standard stochastic control problem into a standard one.} 
\begin{proposition}\label{exitenceastar}
Let $\Gamma\in\mathcal{A}^P_{  \rho-2\lambda}$ and  $A\in\mathcal{A}^C_{  \rho-2\lambda}.$   There exists a bijection between the process  $(Z_t^A)_{t\geq0}$ and the optimal effort $ (A^*_t)_{t\geq 0},$  the bijection is given by
      $$A^*_t=A^*(Z_t^A)=(\frac{h'}{\varphi'})^{-1}(\frac{Z_t^A}{\sigma})\un_{\{Z^A_t>0\}},$$
      or equivalently 
      $$ Z_t^A= (\sigma\frac{h'}{\varphi'})(A^*(Z_t^A))\un_{\{A^*(Z_t^A)>0\}}.$$           
  \end{proposition}
\begin{proof}
We consider the stochastic set $\mathcal{D}=\{(t,\omega), {A^*_t}(\omega)>0\}.$  From  Lemma \ref{lemme01}, we have
 $Z_t^A=\sigma\frac{h'(A^*_t)}{\varphi'(A^*_t)}$. As $h$ and $\varphi$ are increasing (see Assumption \ref{hyp}), $(Z_t^A)_{t\geq0}$ is a positive process.  The function $\frac{h'}{\varphi'}$ being  invertible gives a bijection between $Z_t^A$ and $A_t^*$.\\ 
Now we consider $\mathcal{D}^C:=\{(t,\omega),A^*_t(\omega)=0\}.$ On this set, {from  Lemma \ref{lemme01}},  we have $Z_t^A\leq\sigma\frac{h'(0)}{\varphi'(0)}$. \\
From the definition of the objective function for the public, { and since  the rent is non-negative}, 
$R_t^*=0$ on $\mathcal{D}^C.$  For $(t,\omega)\in\mathcal{D}^C$, the consortium  does not receive any rent from the public and does not provide any effort,  we have from {Equation \eqref{JT2}}
	\begin{eqnarray*}
dJ_t^C(R^*,\tau,\xi,A^*)&=&-\Big(-\lambda J_t^C(R^*,\tau,\xi,A^*)+\frac{Z^{A}_t}{\sigma}\varphi(A^*_t)\Big)dt+Z_t^{A}dW_t.
\end{eqnarray*}
On the other hand, for $(t,\omega)\in\mathcal{D}^C$, $e^{-\lambda t} J_t^C(R^*,\tau,\xi,A^*)$ is constant since $R_t^*=A_t^*=0$ , which yields
\begin{eqnarray*}
dJ_t^C(R^*,\tau,\xi,A^*)&=&\lambda J_t^C(R^*,\tau,\xi,A^*)dt.
\end{eqnarray*}          
The uniqueness of the It\^o decomposition implies $Z^{A}_t=0~a.s.$ This shows
$A_t^*= A^*(Z_t^A)=(\frac{h'}{\varphi'})^{-1}(\frac{Z^A_t}{\sigma})\un_{\{Z_t^A>0\}}$. \pf \\
\end{proof}

\begin{proposition}\label{incentive} { {  Suppose Assumption \ref{hyp}.}
We assume that the generator defined  by  $(t,\omega,z)\mapsto g(t,\omega,z):=\tilde{U}(R_t)-{\tilde{h}(A^*(e^{\lambda t} z))}+ z\frac{\varphi(A^*(e^{\lambda t}z))}{\sigma}$ {satisfies the Assumptions $(H1(\rho))$ and $(H2$) with Lipschitz coefficient $C_g$ satisfying} $C_g^2<\rho$. Then,}
the dynamics of $J^C$ for any incentive compatible contract  $(\Gamma, A^*(Z))$ { (with $\Gamma \in {\cal A}^P_{  \rho-2\lambda}$  and $A^*(Z)\in {\cal A}^C_{  \rho-2\lambda})$} is given by the BSDE with random terminal condition
\begin{eqnarray}\label{Ii1}
dJ_t^C(\Gamma,A^*(Z))&=&-\Big(-{{\lambda}} J_t^C(\Gamma,A^*(Z_t))+U(R_t)+\psi(A^*(Z_t),Z_t)\Big)dt+Z_t dW_t,\nonumber\\
~J_\tau^C(\Gamma,A^*(Z))&=&{  U(\xi)}{\un_{\{\tau<\infty\}}},
\end{eqnarray}
where $A^*(Z)$  is defined in Lemma \ref{lemme01}.
\end{proposition}

{\begin{proof}
For $\Gamma \in {\cal A}^P_{  \rho-2\lambda}$ and $A^*(Z)\in {\cal A}^C_{  \rho-2\lambda}$, we consider the BSDE
\begin{equation}\label{Ii2}
dY_t=-\Big(-{{\lambda}} Y_t+U(R_t)+\psi(A^*(Z_t),Z_t)\Big)dt+Z_t dW_t,~Y_\tau={  U(\xi)} {\un_{\{\tau<\infty\}}}.
\end{equation}
Or equivalently 
\begin{eqnarray}\label{tildey}
d\tilde{Y}_t=-\left(\tilde{U}(R_t)-\tilde h(A^*(e^{\lambda t}\tilde Z_t))+\tilde{Z}_t\frac{\varphi(A^*(e^{\lambda t}\tilde Z_t))}{\sigma}\right)dt+\tilde{Z}_t dW_t,\,~\tilde{Y}_\tau=e^{-\lambda \tau}{  U(\xi)}{\un_{\{\tau<\infty\}}}.
\end{eqnarray}
From the statement of the Proposition, Assumptions $(H1(\rho))$ and $(H2)$ are satisfied. From Proposition \ref{EDSR}, there exists a unique solution
$(\tilde Y,\tilde Z)\in \mathcal{S}^2_\rho(\tau)\times{\cal H}^2_\rho(\tau)$ to the BSDE (\ref{tildey}) or equivalently there exists a unique $(Y,Z)\in {\mathcal{S}^2_{\rho-2\lambda}(\tau)\times\mathcal{H}^2_{\rho-2\lambda}(\tau)}$ which solves (\ref{Ii2}).
On the other hand, as $\left( \esp ^{\pr}\left[ \gamma_{\tau}^{A^*(Z)}\left(\int_0^\tau e^{-\lambda s}(U(R_s)-h(A^*(Z_s)))ds+e^{-\lambda \tau} {  U(\xi)}{\un_{\{\tau<\infty\}}}\right)|\F_t \right]\right)$ is a ($\pr$, $\mathbb{F})$-{ local martingale}, by applying the martingale representation theorem  as in Lemma \ref{lemmeincentive}, there exists a progressively measurable process $(Z_t^{A^*(Z)})_t$ such that: 
\begin{eqnarray*}
dJ_t^C(\Gamma,A^*(Z))&=&-\left(-{{\lambda}} J_t^C(\Gamma,A^*(Z))+U(R_t)-h(A^*(Z_t))+Z_t^{A^*(Z)}\frac{\varphi(A^*(Z_t))}{\sigma}\right)dt\\
&+&Z_t^{A^*(Z)}dW_t,\\
J_\tau^C(\Gamma,A^*(Z^A))&=&{  U(\xi)}{\un_{\{\tau<\infty\}}}.
\end{eqnarray*}
The uniqueness of the solution yields that
$$Z_t= Z_t^{A^*(Z)},~\pr~a.s,~{ \forall t\in  \lbr 0,\tau\lbr}.$$
Moreover, as {$\tilde{\psi}(A^*(e^{\lambda t}\tilde Z_t),\tilde{Z}_t)\geq \tilde{\psi}(A_t,\tilde{Z}_t) ~\pr ~a.s~\forall t\in \lbr 0,\tau\lbr~
  {\forall A\in\mathcal{A}^C_{  \rho-2\lambda}~}$}, by {Corollary \ref{comp2}} we have for all $A\in\mathcal{A}^C_{  \rho-2\lambda}$
$$ J_t^C(\Gamma,A)\leq J_t^C(\Gamma,A^*(Z)),~\forall t\in \lbr 0,\tau\lbr~{} \pr ~a.s.$$\pf
\end{proof}}

{\begin{remarque} According to Proposition \ref{incentive},  for any  contract $\Gamma \in {\cal A}^P_{  \rho-2\lambda}$ and in particular any terminal payment $\xi \in  L^2_{\rho-2\lambda}(\mathcal F_\tau)$ at time $\tau$, $U(\xi)$ can be represented as the terminal value of a stochastic process $(\zeta_t)_{t \geq 0}$ characterized by its initial value $\zeta_0$ and its  diffusion coefficient $Z_t$, its drift  being given by $\left(-{{\lambda}} \zeta_t+U(R_t)+\psi(A^*(Z_t),Z_t)\right)$. Observe that the optimal effort $A^*(Z)$ is then function of the diffusion process $Z$, which is $\mathbb F$-adapted, that is measurable with respect to the filtration generated by the social welfare $X$ of the project. This means that  $Z$ has to be chosen by the public in an optimal way (and indexed on the social welfare $X$) in order  to incentive the consortium to provide the greatest effort. Of course this contract  should be chosen in the set of acceptable contracts to the consortium, that is the initial value $\zeta_0$  should be greater than  the reservation constraint $\underline x$
    $$J_0^C(\Gamma,A^*(Z))= \zeta_0=\esp^{A^*(Z)}\left[\int_0^\tau e^{-{{{\lambda}}s}}(U(R_s)-h(A^*(Z_s)))ds+e^{-{{\lambda}}\tau} U(\xi){\un_{\{\tau<\infty\}}}\right] \geq \underline x.$$
    \end{remarque} }
    
\noindent { Remark also that  the incentive comes from both the rent and the terminal payment. Thus we could not obtain  an explicit form for the terminal payment as a function of the terminal social welfare, as it can be done for a unique lumpsum and  in some particular cases (see \cite{holmstrom1987aggregation}). Here the optimal contract is given in a feedback  form  using a verification theorem, as in Sannikov \cite{sannikov2008continuous}. }

{\begin{corollaire}\label{coro3}
For $\Gamma\in\mathcal{A}^P_{  \rho-2\lambda}$ and  $A^*(Z)\in\mathcal{A}^C_{  \rho-2\lambda}$ the best response, the objective function $J_t^C(\Gamma,A^*(Z))$ satisfies
$$J_t^C(\Gamma,A^*(Z))\geq0~\forall t\in \lbr 0,\tau\lbr,~ \pr ~a.s.$$
\end{corollaire}
\begin{proof}
We consider the contract $\Gamma=(R,\tau,\xi)\in \mathcal{A}^P_{  \rho-2\lambda}$ and $A^*(Z)\in \mathcal{A}^C_{  \rho-2\lambda} $ the best response,  { then as the consortium can guarantee himself non-negative utility by taking an effort equal to zero, the consortium objective function is non-negative, i.e.}, we have 
$$  J_t^C(\Gamma,A^*(Z))\geq0,~\forall t\in \lbr 0,\tau\lbr,~ \pr ~a.s.$$\pf
\end{proof}}

{\begin{remarque}
{ Here are the functions used in our numerical study. As stated before, it is natural to consider an increasing concave function $\varphi$ for the impact of the effort on the social value, and an increasing  convex function $h$ for the cost of the effort.  We take them positive and of exponential form, namely $\varphi(x)=3(1-e^{-\alpha x})$ and  $h(x)=e^{\beta x}-1$, for positive real numbers $\alpha$ and $\beta$. The coefficients $\alpha$ and $\beta$ are chosen such that the two functions have the same magnitude, in order to ensure a reasonable tradeoff between gain and cost.
  This  implies }
the following  generator of the BSDE (\ref{tildey})
\begin{eqnarray*}
 g(t,\omega,z)=\left\{
    \begin{array}{ll}
      \tilde{U}(R_t)+e^{-\lambda t}+\frac{z}{\sigma}-z^{\frac{\beta}{\alpha+\beta}}\bigg(\frac{1}{\sigma}\big(\frac{e^{\lambda t}\alpha }{\beta\sigma}\big)^{\frac{-\alpha}{\alpha+\beta}}+e^{-\lambda t}\big(\frac{e^{\lambda t}\alpha }{\beta\sigma}\big)^{\frac{\beta}{\alpha+\beta}}\bigg),~\mbox{if}~e^{\lambda t}z>\sigma\frac{\beta}{\alpha},\\
       \tilde{U}(R_t),~\mbox{if}~e^{\lambda t}z\leq\sigma\frac{\beta}{\alpha}.
    \end{array}
\right. 
\end{eqnarray*}
We check that the  generator $g$ satisfies the Assumptions  {  $(H1(\rho))$  and  $(H2)$   for $\rho > C^2_g$ since\\
 $\esp\big[\int_0^\tau e^{\rho s}|g(s,\omega,0)|^2ds\big]=\esp\big[\int_0^\tau e^{\rho s}|\tilde{U}(R_s)|^2ds\big] = \esp\big[\int_0^\tau e^{(\rho -2 \lambda) s}|U(R_s)|^2ds \big]<\infty$ as soon as the contract is in $ \mathcal{A}^P_{  \rho-2\lambda}$.}
 {For the numerical study, we will choose such functions.}
\end{remarque}

\noindent We now reformulate the stochastic control problem with $J^C$ as state variable and the contract $\Gamma$ and the { best} effort $A^*(Z)$ as control processes. Usually in the literature (see Sannikov \cite{sannikov2008continuous} and Cvitanic et~al. \cite{cvitanic2018dynamic}), the optimization problem consists in maximizing a certain criterion where the control variables  are given by  $\Gamma$ and {$(Z_t)_{t\geq0}$} which is a standard mixed stopping-regular stochastic control problem. 
In this paper, we keep  the explicit control $(A^*(Z_t))_{t\geq0}$ instead of $(Z_t)_{t\geq0}$,  since $(A^*(Z_t))_{t\geq0}$ represents a physical quantity and thus  it is more quantifiable and interpretable  than the control $(Z_t)_{t\geq0}$ which is a diffusion coefficient.

\section{Hamilton Jacobi Bellman variational inequality} \label{section4}
From now on, we adopt a forward point of view for the dynamics of the consortium objective function which evolves according to  the following forward SDE
\begin{eqnarray}\label{eqJforward}
  dJ_t^C(x,R,\tau,A^*(Z))&=&\left({{\lambda}} J_t^C(x,R,\tau,A^*(Z))-U(R_t)+h(A^*(Z_t))-Z_t\frac{\varphi(A^*(Z_t))}{\sigma}\right)dt+Z_tdW_t, \nonumber \\
 J_0^C(x,R,\tau,A^*(Z))& =& x \geq     \underline x.
\end{eqnarray} 
 \begin{remarque}
 \vspace{-1cm}
 {
  As we choose an initial condition satisfying $J_0^C(x,R,\tau,A^*(Z)) =x \geq     \underline x$,
  the reservation constraint formulated in the maximization problem of the public is satisfied.
  However, we solve the stochastic control problem related to the public on the whole domain i.e on $\R^+$.
  In fact the consortium objective function at time $t$, denoted by $J_t^C(x,R,\tau,A^*(Z))$ could be less than $ \underline x$ although $J_0^C(x,R,\tau,A^*(Z))\geq  \underline x$. }
\end{remarque} 
 Using the characterization of  the incentive compatible  contracts, the optimization problem  of the public can be written as a standard stochastic control problem. The state process is { the consortium objective function $J^C$} whose dynamics is given by (\ref{eqJforward}).
      The control processes  are given by $R$, $\tau$  and $A^*(Z)\in \mathcal{A}^C_{  \rho-2\lambda} $ and
 the value function given by (\ref{Principal}) is formulated as:
\begin{equation}\label{stoppb1}
v(x) :=\sup_{(R,\tau,{{A^*(Z)}})\in {\cal{Y}}
}\esp^{{ A^*(Z)}}_x\left[\int_0^\tau e^{-\delta s}(\varphi(A^*(Z_s))-R_s)ds-e^{-\delta \tau}{   U^{-1}}({J_\tau^C(x,R,\tau,A^*(Z))})\right],
\end{equation}
 where $ \esp^{A^*(Z)}_x$ is the conditional expectation with respect to the initial event { $\{J_0^C(x,R,\tau,A^*(Z))=x\}$ } and
\begin{eqnarray}\label{YY}
    {\cal{Y}} &:=& \Big\{(R,\tau,A^*(Z)) : ~ R\geq 0 ~\ff\mbox{-progressively measurable process s.t.}\nonumber\\
 & & \, \, \, \esp^\pr[{\int_0^\infty}e^{{  (\rho-2\lambda)} s}(U(R_s)^2 \vee R^2_s)  ds]<\infty,~\tau \in {\cal{T}}, { A^*(Z)}\in \mathcal{A}^C_{  \rho-2\lambda}\Big\}.
 \end{eqnarray}
{\noindent The value function $v$  in (\ref{stoppb1}) is defined on $\R^+$ since $J_t^C(\Gamma,A^*(Z))\geq0$ by Corollary \ref{coro3}.}\\
The set
$${\cal{S}} := \{x\geq 0 : { v(x)}\leq -{   U^{-1}}(x)\}$$
is called the stopping region and is of particular interest: whenever the state is in this region, it is optimal to stop the contract immediately. Its complement ${\cal{S}}^c$ is called the continuation region. We apply the dynamic programming principle, which takes the following form: for all stopping time $\eta\in {\cal{T}}, $ we have}{{\small
\begin{eqnarray}\label{DP1}
  {v(x)}&=&\Sup_{(R,\tau,{ A^{*}(Z))}\in {\cal{Y}}}\esp^{{A^*(Z)}}_x\left[\int_0^{\tau\wedge\eta}e^{-\delta s}(\varphi({ A^*(Z_s))}-R_s)ds
    -e^{-\delta\tau} {  {   U^{-1}}}({ J_\tau^C(x,R,\tau,A^*(Z))})\un_{\tau<\eta}\right.\nonumber\\
    &&\quad\quad\quad\quad\quad\quad\quad\quad\quad\quad+\left.e^{-\delta\eta}{ v(J^C_\eta(x, R ,\tau,A^*(Z)))\un_{\eta\leq \tau}} \right], 
\end{eqnarray}
\noindent which is used to derive the Hamilton Jacobi Bellman Variational Inequality (HJBVI) associated to the value function}
\begin{equation}\label{IVHJB}
\min\left\{\delta w(x)-\Sup_{{(r,a)\in \R^{+}\times \R^{+}}}[{\cal L}^{a,r} w(x)+\varphi(a)-r],w(x)+ {  {   U^{-1}}}(x)\right\}=0,\,\, x\in(0,\infty),
\end{equation}
where the second order differential operator ${\cal L}^{a,r}$ is defined by
$${\cal L}^{a,r}w(x):=\frac{1}{2}{(\sigma\frac{h'(a)}{\varphi'(a)})^2}{\un_{\{a>0\}}}w"(x)+ \big( \lambda x-U(r)+h(a) \big) w'(x).$$
The first step consists in giving the boundary condition $v(0)$.
{\begin{lemme}\label{lemme1}
The function $v$ defined in (\ref{stoppb1}) satisfies
\begin{equation}\label{CB}
v(0)=0
\end{equation}
\end{lemme}
\begin{proof}
We have $J_0^C(0,R,\tau,A^*(Z))=0$. From Corollary \ref{coro3},  we have  $J_t^C(\Gamma,A^*(Z))\geq0$, so for all  $t\geq0,~J_t^C(0,R,\tau,A^*(Z))$ is non-negative almost surely.  {  Since $J_t^C$ satisfies 
  the SDE (\ref{eqJforward}) and starts from an initial value equal to zero, then  for obtaining a non-negative  solution, we must have  that the infinite variation part of  its dynamics should be zero, that is $ Z_t=0~dt \otimes d\pr\, a.e.$.} From  the  bijection between the process $(Z_t^A)_{t\geq0}$ and the optimal effort $ (A^*_t)_{t\geq 0}$  (see Proposition \ref{exitenceastar}), we must have $A^*_t=0~dt\otimes d\pr\, a.e.$  Therefore, from the definition of the value function (\ref{stoppb1}), it is optimal for the public to choose $R_t^*=0~dt\otimes d\pr\, a.s$, since $U(0)=0$ by Assumption \ref{hyp}. As the drift of SDE (\ref{eqJforward}) is equal to $0$, we must have $\tau=0~\pr a.s.$ This shows $v(0)=0.$
\end{proof}
\pf
\subsection{Verification theorem}
In order to provide  the   verification theorem, 
a first lemma shows that the public value function satisfies a linear growth condition.
{\begin{lemme}  There exists a positive constant  $K$  such that  
\begin{equation}\label{eqcroissance}  \mbox{  for all }x \geq 0, |v(x)| \leq K +{  U^{-1}}(x).
\end{equation}
\end{lemme}}
{\begin{proof}
From the definition (\ref{stoppb1}) of the value function $v$  and for all $(R,0,A^*(Z))\in\mathcal{Y}$ we have $v(x)\geq - {  U^{-1}}(J_0^C(x,R,0,A^*(Z)))=-{  U^{-1}}(x)$.
Moreover, since $J_\tau^C(x,R,\tau,A^*(Z))\geq0~\pr~a.s,$ 
\begin{eqnarray*}
v(x)&\leq& \sup_{(R,\tau,{{A^*(Z)}})\in {\cal{Y}}
}\esp^{{ A^*(Z)}}_x\left[\int_0^\tau e^{-\delta s}(\varphi(A^*(Z_s))-R_s)ds\right]\\
&\leq& \sup_{y\geq 0}\frac{1}{\delta}\{ \varphi\circ h^{-1}\circ (U(y))-y  \}=:K.
\end{eqnarray*}
This implies that inequality \eqref{eqcroissance}
\pf
\end{proof}}
\noindent { The  verification  theorem  stated below  specifies the solution of the HJB Variational Inequality  (\ref{IVHJB}) on respectively the continuation and stopping regions.}
\begin{theoreme} [Verification Theorem] ${}$\\ \label{verification}
 {  We suppose that there exist   a constant $\hat b > 0$ and  a continuous function $w: \mathbb R^+ \longrightarrow \R$ satisfying \\
  (i) w(0)=0, $w\in C^2([0,\hat b))$ satisfying the  growth condition \eqref{eqcroissance},\\
  (ii) $w>-U^{-1}$ on $(0,\hat b)$ and  $w=-U^{-1}$ on $[\hat b,\infty)$\\
  (iii) $\delta w(x)-\Sup_{(r,a)\in\R^+\times\R^+}\{{\cal{L}}^{a,r}w(x)+\varphi(a)-r\}=0$ for all $x \in (0,\hat b)$.\\
(iv) $\delta (-U^{-1}(x))-\Sup_{(r,a)\in\R^+\times\R^+}\{{\cal{L}}^{a,r}(-U^{-1}(x))+\varphi(a)-r\} \geq 0$ for all $x \in  [\hat b,\infty)$.\\ }
\noindent We also assume that 
\begin{equation}\label{hyp3}
{{\Sup_{(R,\tau,A^*(Z))\in\mathcal Y}\esp \left[|e^{-\delta\tau}{   U^{-1}}(J_\tau^C(x,R,\tau,A^*(Z)))|^2 \right]<\infty},}
\end{equation} 
Then we have:
\begin{itemize}
\item [(1)]   $w(x)\geq v(x)$ for  any { $x\geq 0$},
\item [(2)] Suppose there exists two  measurable non-negative functions $(a^*,r^*)$ defined on $\mathbb R^+$ s.t. 
\begin{eqnarray}\label{conop} 
  \Sup_{(r,a)\in\R^+\times\R^+}\{{\cal{L}}^{a,r}w(x)+\varphi(a)-r\}={\cal{L}}^{a^*(x),r^*(x)}w(x)+\varphi(a^*(x))-r^*(x), \, \ { x \in (0,\hat b)},
 \end{eqnarray}   
 and that the SDE 
 $$ dJ_t^C=\left({{\lambda}} J_t^C-U(r^*(J_t^C))+h(a^*(J_t^C))-Z_t\frac{\varphi(a^*(J_t^C))}{\sigma}\right)dt+Z_tdW_t, \, \quad {  J_0^C\geq \underline{x}}$$
  admits a unique solution $\widehat{J_t^C}$.
   We define 
 \begin{equation}\label{tauverification}
\tau^* := ~\inf\{t\geq 0~:~ w(\widehat{J_t^C})\leq -{   U^{-1}}(\widehat{J_t^C})\}
\end{equation}
 and we assume that $( r^*(\widehat{J^C}) ,\tau^*, a^*(\widehat{J^C}))$ lies in  ${\cal{Y}}$ and {${\esp^\pr[e^{  (\rho-2\lambda)\tau^*}\widehat{J_{\tau^*}^C}^{2}\un_{\{\tau^*<\infty\}}]}<\infty.$}\\
Then  we have
\begin{itemize}
\item [(a)] 
$w=v,$   and $\tau^*$ is an optimal stopping time for the problem (\ref{stoppb1}).\\
\item [(b)] The optimal rent is given by $r^*(x)=(U')^{-1}(-\frac{1}{w'(x)})\un_{w'(x)<0}$ { for all $x\in(0,\hat b)$}.
\end{itemize}
\end{itemize}

\end{theoreme}

\noindent { Remark that  if  $J^C_0=x\in(0,\hat b)$ (that is the initial value of $J^C$ is in the continuation region), then the optimal stopping time is also characterized by  $\tau^* = ~\inf\{t\geq 0~:~ \widehat{J_t^C} = 0 \mbox{ or } \hat b\}$. We refer to Section \ref{num:result} for a numerical illustration.} 

\begin{proof}\\
(1) Let $x >0$ and $n\in\N,$ for an admissible contract $(R,\tau,A^*(Z))\in\mathcal Y$, 
we denote\\
$$\tau_n~:=~ \tau\wedge \inf\{{t\geq0}: |w'(J_t^C(x,R,\tau,A^*(Z))){ \sigma\frac{h'(A^*(Z_t))}{\varphi'(A^*(Z_t))}}{\un_{\{A^*(Z_t)>0\}}}	| > n\}.$$
{ From (i)-(ii), we have $w$ is continuous on $\R_+$, $w\in C^2([0,\hat b))$ and  $w=-U^{-1}\in C^2([\hat b,\infty))$, then $w$ is continuous and piecewise $C^2$ on $\R_+$.
Applying the generalized It\^o's formula (see Krylov \cite{kry}, Theorem 2, p. 124) between time $0$ en $\tau_n$ to the process 
$\left(e^{-\delta t}w(J_t^C(x,R,\tau,A^*(Z)))\right)_{t\geq 0}$ }
\begin{eqnarray*}
w(x)&=&e^{-\delta\tau_n}w(J_{\tau_n}^C(x,R,\tau,A^*(Z)))\\
&-&\int_0^{\tau_n}e^{-\delta s} \left( -\delta w(J_s^C(x,R,\tau,A^*(Z)))
+{\cal{L}}^{A^*(Z_s),R_s}w(J_{s}^C(x,R,\tau,A^*(Z))) \right)ds\\
&+&\int_0^{\tau_n}e^{-\delta s}w'(J_s^C(x,R,\tau,A^*(Z)))\sigma\frac{h'(A^*(Z_s))}{\varphi'(A^*(Z_s))}{\un_{\{A^*(Z_s)>0\}}}dW_s^{A^*(Z)}.
\end{eqnarray*}
Taking the expectation, we obtain
\begin{eqnarray}\label{Ito2term}\nonumber
\hspace*{-3cm}w(x)&=&\esp^{{ A^*(Z)}}\Bigg[e^{-\delta \tau_n}w(J_{\tau_n}^C(x,R,\tau,A^*(Z)))-\int_0^{\tau_n}e^{-\delta s}[-\delta w(J_{s}^C(x,R,\tau,A^*(Z))) \\ \nonumber
&+& {\cal{L}}^{A^*(Z_s),R_s}w(J_{s}^C(x,R,\tau,A^*(Z)))]ds\Bigg]\\ \nonumber
&\geq& \esp^{ A^*(Z)}\left[e^{-\delta\tau_n}w(J_{\tau_n}^C(x,R,\tau,A^*(Z)))+ \int_0^{\tau_n}e^{-\delta s}(\varphi(A^*(Z_s))-R_s)ds\right]\\  &=&\esp\left[\gamma_{\tau_n}^{A^*(Z)}e^{-\delta\tau_n} w(J_{\tau_n}^C(x,R,\tau,A^*(Z)))+\gamma_{{\tau_n}}^{A^*(Z)} \int_0^{\tau_n} e^{-\delta s}(\varphi(A^*(Z_s))-R_s)ds\right]
\end{eqnarray}
{ where the  inequality is obtained by using Assumptions (iii)-(iv)}, and the last equality is obtained by using the Bayes formula.\\
We first study the limit of the second term of \eqref{Ito2term}. Since $\hat{p}\in(2,\infty),$ there exists a unique $\varepsilon\in(0,1)$ such that $\hat{p}=2\frac{1+\varepsilon}{1-\varepsilon}$. We define $p:=1+\varepsilon\in(1,2).$ We fix $q_1:=\frac{2}{2-p}$ and let $q_2$ be the conjugate of $q_1$ i.e. $q_2=\frac{2}{p}$. We have
{\small$$ \esp\left[\big|\gamma^{ A^*(Z)}_{\tau_n}\int_0^{\tau_n}e^{-\delta s}(\varphi(A^*(Z_s))-R_s)ds\big|^p\right]\leq\esp\left[|\gamma^{ A^*(Z)}_{\tau_n}|^{pq_1}\right]^{\frac{1}{q_1}}\esp\left[\big|\int_0^{\tau_n}e^{-\delta s}(\varphi(A^*(Z_s))-R_s)ds\big|^{pq_2}\right]^{{\frac{1}{q_2}}}.$$
 Thanks to our choice of $p$, $q_1$ and $q_2$, we have $pq_1=\hat p$ and $pq_2=2$. {By using Cauchy Schwartz, we obtain
 {\footnotesize 
 \begin{eqnarray*}
 \esp\left[\big|\gamma^{ A^*(Z)}_{\tau_n}\int_0^{\tau_n}e^{-\delta s}(\varphi(A^*(Z_s))-R_s)ds\big|^p\right]&\leq&\esp\left[|\gamma^{ A^*(Z)}_{\tau_n}|^{\hat{p}}\right]^{\frac{1}{q_1}}\esp\left[\int_0^\infty e^{-\delta s}ds\int_0^{\infty}e^{-\delta s}(\varphi(A^*(Z_s))-R_s)^{2}ds\right]^{{\frac{1}{q_2}}}\\
 &\leq&\frac{1}{\delta}\esp\left[|\gamma^{ A^*(Z)}_{\tau_n}|^{\hat{p}}\right]^{\frac{1}{q_1}}\esp\left[\int_0^{\infty}e^{-\delta s}(\varphi(A^*(Z_s))-R_s)^{2}ds\right]^{{\frac{1}{q_2}}}.
 \end{eqnarray*}}
 }By using the definition of the set $\mathcal{A}^{\hat p}$ and the integrability conditions on $\mathcal{A}^C_{  \rho -2 \lambda}$ and $\mathcal{A}^P_{  \rho -2 \lambda}$, we obtain
 \begin{eqnarray*}
&& \hspace*{-1cm}\Sup_{n\in\N}\esp\left[|\gamma^{ A^*(Z)}_{\tau_n}\int_0^{\tau_n}e^{-\delta s}(\varphi(A^*(Z_s))-R_s)ds|^p\right]\\ 
&\leq&C\Sup_{\tau\in\mathcal{T}}\esp\left[|\gamma_\tau^{ A^*(Z)}|^{\hat{p}}\right]^{\frac{2-p}{2}}
 \Bigg(\esp\left[\int_0^\infty e^{-  \delta s}|\varphi(A^*(Z_s))|^{2}ds\right]
+\esp\left[\int_0^\infty e^{-\delta s}|R_s|^{2}ds\right]\Bigg)^{\frac{p}{2}}\\
&<&\infty.
\end{eqnarray*}}
So, we have $\Sup_{n\in\N}\esp \left[|\gamma^{ A^*(Z)}_{\tau_n}\int_0^{\tau_n}e^{-\delta s}(\varphi(A^*(Z_s))-R_s)ds|^{ p}\right]<\infty,$ for this fixed $p>1$, therefore\\
 $\bigg(\gamma^{ A^*(Z)}_{\tau_n}\int_0^{\tau_n} e^{-\delta s}(\varphi(A^*(Z_s))-R_s)ds\bigg)_{n\in\N}$ is uniformly integrable under $\pr$. This implies  the convergence in  $L^1(\pr)$ (see Theorem $A.1.2$ in \cite{pham2009continuous}) and  we may pass to the limit as  $n\rightarrow\infty,$ 
\begin{equation}\label{lim1}
\Lim_{n\rightarrow\infty}\esp \left[\gamma^{A^*(Z)}_{\tau_n}\int_0^{\tau_n}  e^{-\delta s}(\varphi(A^*(Z_s))-R_s)ds\right]=\esp \left[\gamma_{\tau}^{A^*(Z)} \int_0^{\tau}  e^{-\delta s}(\varphi(A^*(Z_s))-R_s)ds\right].
\end{equation}
{The same methodology applies for studying the limit of  the first term of \eqref{Ito2term}.}
 As $w$ satisfies the  growth condition \eqref{eqcroissance}, we have  with $q_1=\frac{2}{2-p}$ and its conjugate  $q_2=\frac{2}{p}$
\begin{eqnarray*}
\esp\left[|\gamma_{\tau_n}^Ae^{-\delta\tau_n}w(J_{\tau_n}^C(x,R,\tau,A^*(Z)))|^p\right]&\leq&\esp\left[|\gamma_{\tau_n}^A e^{-\delta\tau_n}(K+{   U^{-1}}(J^C_{\tau_n}(x,R,\tau,A^*(Z))))|^p\right]\\
&& \hspace*{-5cm} \leq {(K^p\vee1)}\esp\left[|\gamma_{\tau_n}^A+\gamma_{\tau_n}^A e^{-\delta\tau_n}{   U^{-1}}(J^C_{\tau_n}(x,R,\tau,A^*(Z))) |^p\right]\\
&& \hspace*{-5cm} \leq {(K^p\vee1)}\left( \esp[(\gamma_{\tau_n}^A)^p]+\esp\left[|\gamma_{\tau_n}^A|^{pq_1}\right]^{\frac{1}{q_1}}\esp\left[|{e^{-\delta\tau_n}}{   U^{-1}}(J^C_{\tau_n}(x,R,\tau,A^*(Z)))|^{pq_2}\right]^{\frac{1}{q_2}}\right).
\end{eqnarray*}
{Inequality} (\ref{hyp3}) and the  integrability conditions on $\mathcal{A}^{\hat p}$ imply
{\small \begin{eqnarray*}
&& \hspace*{-1cm} \sup_{n\in\N}\esp\left[|\gamma_{\tau_n}^{A^*(Z)}e^{-\delta \tau_n}w (J_{\tau_n}^C(x,R,\tau,A^*(Z)))|^p\right]\\
& \leq & {(K^p\vee1)}\left(\Sup_{\tau\in\mathcal{T}}\esp\left[(\gamma_\tau^{A^*(Z)})^p\right]+\Sup_{\tau\in\mathcal{T}}\esp\left[(\gamma_\tau^{A^*(Z)})^{\hat{p}}\right]^\frac{1}{q_1}\Sup_{(R,\tau,A^*(Z))\in\mathcal Y}\esp\left[|e^{-\delta\tau}{   U^{-1}}(J^C_\tau(x,R,\tau,A^*(Z)))|^{2}\right]^\frac{1}{q_2}\right)\\
&<&\infty.
\end{eqnarray*}}
Then, we may pass to the limit as $n\rightarrow\infty$ 
\begin{equation}\label{lim2}
\lim_{n\rightarrow+\infty}\esp\left[e^{-\delta\tau_n}\gamma_{\tau_n}^{A^*(Z)} w(J_{\tau_n}^C(x,R,\tau,A^*(Z)))\right]=\esp\left[e^{-\delta\tau}\gamma_{\tau}^{A^*(Z)} w(J_{\tau}^C(x,R,\tau,A^*(Z)))\right].
\end{equation}
By (\ref{lim1}) and (\ref{lim2}), together with  \eqref{Ito2term},  we have
$$w(x)\geq \esp^{{ A^*(Z)}}\left[ \int_0^{\tau}e^{-\delta s}(\varphi(A^*(Z_s))-R_s)ds+e^{-\delta\tau}w(J_{\tau}^C(x,R,\tau,A^*(Z)))\right].$$
Since $w(J_{\tau}^C(x,R,\tau,A^*(Z)))\geq -{   U^{-1}}(J_\tau^C(x,R,\tau,A^*(Z)))$, this leads to
$$w(x)\geq \esp^{{ A^*(Z)}}\left[ \int_0^{\tau}e^{-\delta s}(\varphi(A^*(Z_s))-R_s)ds-e^{-\delta\tau}{   U^{-1}}(J_{\tau}^C(x,R,\tau,A^*(Z)))\right].$$
By taking the supremum, we obtain that for all $x>0$
$$w(x)\geq \Sup_{(R,\tau,A^*(Z))\in{\cal{Y}}}\esp^{{ A^*(Z)}}\left[ \int_0^{\tau}e^{-\delta s}(\varphi(A^*(Z_s))-R_s)ds-e^{-\delta\tau}{   U^{-1}}(J_{\tau}^C(x,R,\tau,A^*(Z)))\right]=v(x).$$
{  From (i), we also have  $w(0)=0=v(0)$, thus  $w\geq v$ on $\mathbb R^+$. }\\
\noindent (2)-a.  { We fix  $x\in(0,\hat b)$.} We now consider the feedback control $( r^*(\widehat{ J^C}) ,\tau^*, a^*(\widehat{ J^C})) $ which is assumed to be in ${\cal{Y}}$. We denote
$$\tau_n~:=~\tau^*\wedge \inf\{{t\geq0}:{ |w'(\widehat{ J_t^C)}}{ \sigma\frac{h'(a^*(\widehat{ J_t^C)})}{\varphi'(a^*(\widehat{ J_t^C)})}}{\un_{\{a^*(\widehat{ J_t^C})>0\}}}|> n\}.$$
Observe that  $w(\widehat{ J_t^C})>-\widehat{ J_t^C}$ on $\lbr0,\tau_n{\lbr}\subset\lbr0,\tau^*\lbr$. 
Then on $\lbr0,\tau_n\lbr$, by (\ref{conop} ) and { (iii)}, we have 
$$\delta w({\widehat J^C}_t)- [{\cal L}^{a^*({\widehat J^C}_t),r^*({\widehat J^C}_t)} w({\widehat J^C}_t)+\varphi(a^*({\widehat J^C}_t))-r^*({\widehat J^C}_t)]=0.$$
Therefore
\begin{eqnarray*}
w(x)&=&\esp^{{ a^*(\widehat{ J^C})}}\left[e^{-\delta \tau_n}w({\widehat J}_{\tau_n}^C)-\int_0^{\tau_n}e^{-\delta s}\big(-\delta w({{\widehat J^C}}_s)+{\cal{L}}^{a^*({ {\widehat J^C}}_s),r^*({\widehat J^C}_s)}w({{\widehat J^C}}_s)\big)ds\right]\\
&=&\esp^{{ a^*(\widehat{ J^C})}}\left[e^{-\delta \tau_n}w({\widehat J^C}_{\tau_n})+\int_0^{\tau_n}e^{-\delta s}(\varphi(a^*({\widehat J^C}_s))-r^*({\widehat J^C}_s))ds\right]\\
&=&\esp\left[e^{-\delta\tau_n}\gamma_{\tau_n}^{ a^*(\widehat{ J^C})}w(\widehat{ J^C_{\tau_n}})+\gamma_{\tau_n}^{ a^*(\widehat{ J^C})}\int_0^{\tau_n}e^{-\delta s}(\varphi(a^*(\widehat{ J_s^C}))-r^*({\widehat J^C}_s))ds\right].
\end{eqnarray*}
Similarly to  (1), we show that {$\bigg(\gamma_\tau^{a^*(\widehat{ J^C})}\int_0^{\tau_n} e^{-\delta s}(\varphi(a^*({\widehat J^C}_s))-r^*({\widehat J^C}_s))ds\bigg)_{n}$}  and {$\left(\gamma_{\tau_n}^{ a^*(\widehat{ J^C})}w(\widehat{ J^C_{\tau_n}})\right)_{n}$} are uniformly integrable under $\pr$. Passing to the limit as
$n\rightarrow\infty$,  $\tau_n \rightarrow \tau^*$ a.s.  and since $w({\widehat J^C}_{\tau^*})=-{   U^{-1}}({\widehat J^C}_{\tau^*})$, we obtain 
\begin{eqnarray*}
w(x)&=&\esp^{{ a^*({\widehat J^C})}}\left[\int_0^{\tau^*} e^{-\delta s}(\varphi(a^*({\widehat J^C}_s))-r^*({\widehat J^C}_s))ds-e^{-\delta \tau^*}{   U^{-1}}({\widehat J^C}_{\tau^*})\right]\\
&=&J^P_0( r^*({\widehat J^C}),\tau^*, a^*({\widehat J^C}             ) )\\
&{\leq}& v(x).
\end{eqnarray*}
{ As $w(0)=v(0)=0$, we conclude that $w=v$ on $[0, \hat b)$ and $(r^*({\widehat J^C}),\tau^*, a^*({\widehat J^C}  ))$ is an optimal feedback control. If $x \geq \hat b$, then $\tau^*=0$ which means that we are  in the stopping region, where $v(x)=-U^{-1}(x)=w(x)$.}  \\
(2)-b. { For a fixed $x$ in $(0,\hat b)$}, we maximize the function 
$$ f(x,.) :r\mapsto -w'(x)U(r)-r.$$
When $w'(x)\geq 0$,  the function $f(x,.)$ is non-increasing and the optimum is achieved for  $r=0.$\\
Otherwise, the function $f(x,.)$ is concave and the optimal rent is given by
 $r^*(x)=\argmax_{r} (f(x,r)).$ Therefore
$$r^*(x)=(U')^{-1}(\frac{-1}{w'(x)})\un_{w'(x)<0}.$$
Furthermore, since $(U')^{-1}(\infty)=0$, $x\rightarrow r^*(x)$ is continuous even at zero points of the function $w'$.\pf
\end{proof}

%

\begin{remarque}
  To study the BSDE (\ref{Ii1}), we need only $L^2$-integrability conditions on the terminal condition  i.e. $J_\tau^C(\Gamma,A^*(Z))=\xi$. However, to prove the verification theorem, we need to strengthen this integrability condition by assuming  a boundedness condition in $L^2$.
\end{remarque}

\begin{remarque}
  Since we do not assume that the controls are in a bounded domain,  the regularity of the principal value function $v$ is not obvious.
Nevertheless, one could characterize the principal value function $v$
as the unique viscosity solution of the associated HJBVI in the class of sublinear functions.
\end{remarque}

\subsection{Beyond the constant volatility case }\label{Rksigmamap}
In all the paper, we assumed that the volatility of the project social value is constant.
One could generalize to a positive map $\sigma$, $\ff$-progressive. We assume also, that the map 
$(x,a)\longrightarrow \frac{\varphi(a)}{\sigma(x)}$ is bounded by a positive constant. This assumption replaces the boundedness of
$\varphi$. In this case, the best response of the agent  depends on two variables and becomes $A^*(x,z)=(\frac{h'}{\varphi'})^{-1}(\frac{z}{\sigma(x)})$. Then the state process is now two-dimensional, the first component being the 
 social value of the project and  the second component being the consortium objective function. They evolve according to the following forward SDEs
 \begin{eqnarray*}
   dX^x_t&=&\sigma (X^x_t)dW_t,\,\,X^x_0=x,\\
  dJ_t^{C,y}&=&-\left(-{{\lambda}} J_t^{C,y}+U(R_t)-h(A^*(X^x_t,Z_t))+Z_t\frac{\varphi(A^*(X^x_t,Z_t))}{\sigma(X^x_t)}\right)dt+Z_tdW_t, 
  J_0^{C,y}=y.
\end{eqnarray*} 
 The control processes are given by $R$, $\tau$  and $A^*(X^x,Z)\in\mathcal{A}^C_{  \rho -2 \lambda}$, and
 the principal value function is given by:
\begin{equation*}
v(x,y) :=\sup_{(R,\tau,{{A^*(X^x,Z)}})\in {\cal{Y}}
}\esp^{{ A^*(X^x,Z)}}\left[\int_0^\tau e^{-\delta s}(\varphi(A^*(X^x_s,Z_s))-R_s)ds-e^{-\delta \tau}{   U^{-1}}({ J_\tau^{C,y}})\right],
\end{equation*}
 and
{\begin{eqnarray*}
    {\cal{Y}} &:=& \Big\{(R,\tau,A^*(X^x,Z))~ R\geq 0, R  ~  \ff\mbox{-progressively measurable process s.t.}\nonumber\\
 & & \, \, \, \esp^\pr[{\int_0^\infty}e^{{{  (\rho -2 \lambda)}} s}(U(R_s)^2 \vee R^2_s)  ds]<\infty,~\tau \in {\cal{T}}, { A^*(X^x,Z)}\in \mathcal{A}^C_{  \rho -2 \lambda}\Big\}.
 \end{eqnarray*}}
The associated Hamilton Jacobi Bellman Variational Inequality (HJBVI) associated to the value function is
\begin{equation*}
\min\left\{\delta w(x,y)-\Sup_{{(r,a)\in \R^{+}\times \R^{+}}}[{\cal L}^{a,r} w(x,y)+\varphi(a)-r],w(x,y)+{   U^{-1}}(y)\right\}=0,\,\, (x,y)\in\R \times (0,\infty),
\end{equation*}
where the second order differential operator ${\cal L}^{a,r}$ is defined by
\begin{eqnarray*}
{\cal L}^{a,r}w(x,y)&:=&\frac{1}{2}\sigma^2(x)\frac{\partial^2 w}{\partial x^2}(x,y)+
\frac{1}{2}{(\sigma(x)\frac{h'(a)}{\varphi'(a)})^2}{\un_{\{a>0\}}}\frac{\partial^2 w}{\partial y^2}(x,y)\\
&+& \frac{1}{2} \sigma^2(x)\frac{h'(a)}{\varphi'(a)}{\un_{\{a>0\}}} \left( \frac{\partial^2 w}{\partial x\partial y}(x,y) +\frac{\partial^2 w}{\partial y\partial x}(x,y) \right)\\
&+&\big({\lambda} y-U(r)+h(a) \big)\frac{\partial w}{\partial y}(x,y).
\end{eqnarray*}

\section{Numerical study} \label{section5}
We approximate numerically the solution of the HJBVI   (\ref{IVHJB}) by using a policy iteration algorithm named Howard algorithm.
 The numerical approximation of the  solution of \eqref{IVHJB} consists in three steps:
\begin{enumerate}
\item Reduction to a bounded domain.  We have to replace    $[0,\infty)$  by a bounded domain $[0,\overline{x}]$.   Since the behavior of the HJB solution at $\infty$  is known, $v(x)=-{   U^{-1}}(x)$ for $x$ large enough, we propose this  relevant artificial boundary condition. The  choice  of the boundary $\overline{x}$ is empirical 	and the robustness is studied by varying $\overline{x}$.
 \item We use  finite difference approximations to discretize the variational inequality (\ref{IVHJB}).
 \item  We use Howard algorithm (see Howard \cite{howard1960dynamic}) to solve the discrete equation.
 \end{enumerate}
Steps 2 and 3 are detailed below.

\subsection{Numerical scheme}

\paragraph{Finite difference approximations}
 Let  $\Delta$ be the finite difference step on the state coordinate and  $x^{\Delta}=(x_i)_{ i=1,...,N}$, $x_i=i\Delta$,
be the points of the grid $\Omega_\Delta.$
 The HJBVI \eqref{IVHJB} is discretized by replacing the first and second derivatives of $v$ with  the following approximations
\begin{eqnarray*}
 v'(x) &\simeq& \left\{ \begin{array}{ll}
\frac{v(x+\Delta)-  v(x)}{\Delta} & \textrm{ if $-\lambda x-h(a)+U(r)\geq 0,$}\\
\frac{v(x)-  v(x-\Delta)}{\Delta} & \textrm{ if not,} 
\end{array} \right.\\
 v{''}(x) &\simeq & \frac{ v(x+\Delta)- 2 v(x)+v(x-\Delta)}{\Delta^2},\\
v(0)&=&0,\\
v(\overline{x})&=&-{   U^{-1}}(\overline{x}).
\end{eqnarray*}
 This leads to the system of $(N-1)$ equations with  $(N-1)$ unknowns $(v^\Delta(x_i))_{i=1,...N-1}$:
 \begin{equation}\label{inequalities}
 \min \left[\Inf_{(r,a)\in\R^+\times \R^+}[A^{\Delta,(a,r)}v^\Delta(x_i)+B^{\Delta,(a,r)}],{v^\Delta(x_i)}+{   U^{-1}}({{x}^\Delta}) \right]=0
 \end{equation}
 where
$B^{\triangle,(a,r)}$ is given by
\begin{displaymath}
B^{\triangle,(a,r)}=
\left( \begin{array}{cccccc}
-\varphi(a)+r  \\
-\varphi(a)+r \\
\vdots\\
 -\varphi(a)+r \\
-\varphi(a)+r+(\frac{\beta^+((N-1)\Delta)}{\Delta}+\frac{\alpha((N-1)\Delta)}{\Delta^2})v(\overline{x})\\
\end{array} \right)
\end{displaymath}
{$x^\Delta$  is a vector of size $N-1$,  $x_i^\Delta=x_i$ and}
the tridiagonal matrix $A^{\Delta,(a,r)}$ is defined as follows:
$$[A^{\Delta,(a,r)}]_{i,i-1}=\frac{\beta^-(x_i)}{\Delta}+\frac{\alpha(x_i)}{\Delta^2}~{\mbox{for}~2\leq i\leq N-1};$$
$$[A^{\Delta,(a,r)}]_{i,i}=\gamma(x_i)-\frac{|\beta(x_i)|}{\Delta}-2\frac{\alpha(x_i)}{\Delta^2}~\mbox{for}~{1\leq i\leq N-1};$$
$$[A^{\Delta,(a,r)}]_{i,i+1}=\frac{\beta^+(x_i)}{\Delta}+\frac{\alpha(x_i)}{\Delta^2}~{\mbox{for}~1\leq i\leq N-2};$$
with  $\beta^+(x)=\Max(\beta(x),0)$, $\beta^-(x)=\Max(-\beta(x),0)$ and
\begin{eqnarray*}
\gamma(x)&=&\delta,\\
\beta(x)&=&-\lambda x-h(a)+U(r),\\
\alpha(x)&=&-\demi (\frac{\sigma h'(a)}{\varphi'(a)})^2{ \un_{\{a>0\}}}.
\end{eqnarray*}
The system of $(N-1)$ inequalities (\ref{inequalities}) can be solved by Howard's algorithm. We describe below this algorithm.

\paragraph{The Howard algorithm}
To solve  equation \eqref{inequalities}, we use  Howard's algorithm. It consists in computing iteratively two sequences $((a^n(x_i),r^n(x_i))_{i=1,...N-1})_{n \geq 1}$ and $((v^{\Delta,n}(x_i))_{ i=1,...N-1})_{n \geq 1}$ (starting from $v^{\Delta,1})$ as follows:
\begin{itemize}
\item[\pf] Step $2n-1$.  {To the vector $v^{\Delta,n},$ we compute }a strategy 
{
$$(a^n,r^n)\in \arg\min_{a,r}\left\{A^{\Delta,(a,r)}v^{\Delta,n}+B^{\Delta,(a,r)}
\right\}.$$}
\item[\pf] Step $2n$. From  the strategy $(a^n,r^n)$, we compute a partition $(D^n_1\cup D^n_2)$ of  $\mathbb R^+$  defined by
\begin{equation*}
\begin{array}{rcl}
A^{\Delta,(a^n,r^n)}v^{\Delta,n}+B^{\Delta,(a^n,r^n)}&\leq&{v^{\Delta,n}}+{   U^{-1}}(x^{\Delta}),~\mbox{on}~D_1^n,\\
A^{\Delta,(a^n,r^n)}v^{\Delta,n}+B^{\Delta,(a^n,r^n)}&>&{v^{\Delta,n}}+{   U^{-1}}(x^{\Delta}),~\mbox{on}~D_2^n.
\end{array}
\end{equation*}
The solution $v^{\Delta,n+1}$ is obtained by solving two linear systems
$$A^{\Delta,(a^n,r^n)}v^{\Delta,n+1}+B^{\Delta,(a^n,r^n)}=0,~\mbox{on}~D_1^n,$$
and 
$${v^{\Delta,n+1}}+{   U^{-1}}(x^{\Delta})=0,~\mbox{on}~D_2^n.$$
\item[\pf] If $|v^{\Delta,n+1}-v^{\Delta,n}|\leq \varepsilon$, stop, otherwise, go to step $2n+1.$
\end{itemize}

\begin{remarque}
Barles and Souganidis \cite{barles1991convergence} proved that if a numerical scheme satisfies the monotonicity, the consistency and the stability, then the numerical scheme solution converges to the viscosity solution of the HJBVI,
by relying on
the PDE characterization of $v$ and the strong comparison principle for the HJBVI.
\end{remarque}

\subsection{Numerical results}\label{num:result}
 {  
\noindent For the numerical implementation,  we choose the following functions for  $\varphi$ (the impact of the effort on the social welfare), $h$ (the cost of effort) and the  consortium's utility $U$:\\
 $\varphi(x)=3(1- \exp(-\alpha x))$, $h(x)=\exp(\beta x)-1$ and $U(x)={x}^{\frac{1}{4}}$; $\alpha=0.1$ and $\beta=0.1$.\\ 
 The preference parameters for the public and the consortium are  respectively $\delta=0.08$ and $\lambda=0.2$.
 We study numerically the impact of the volatility by varying\footnote{{  The size of $\sigma$ is chosen such that $\sigma \sqrt{t}$ (the standard deviation of the noise $\sigma W_t$ in the social value $X_t$)  and $\int_0^t \phi(A_s)ds$  (the drift of the social value $X_t$) have the same magnitude. For $t=1$ i.e. $t=1 $ century and
  $\phi(x)=3(1-e^{-\alpha x})\leq 3$, one could choose $\sigma$ between 1 and 2.}}  $\sigma$  : $\sigma = 1.55,~1.65$ or $2$.
 We start from $v(0)=0$ and we take $\overline{x}=0.8$.\\
 Figure \ref{FIG0} represents the function value on $ [0,\bar{x}],$ for $\sigma=1.55$.
\begin{figure}[!h]
\begin{center}
\centering \includegraphics[width=12cm,height=10cm]{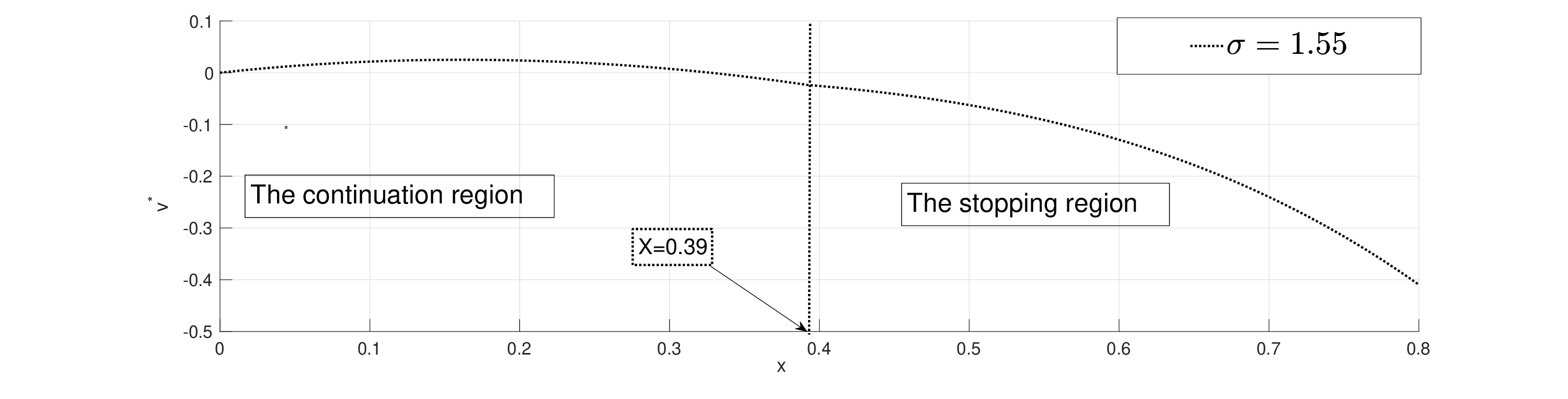}
\caption{ The value function in $[0,\bar{x}].$ }
\label{FIG0}
\end{center}
	\end{figure}
We observe that the value function is concave,  in accordance with  Sannikov \cite{sannikov2008continuous} and Possamaï et al. \cite{possamai2020there}.
For $\sigma=1.55$, the continuation region is $(0,0.39)$ on which the value function is strictly concave, then it is  equal to $-{   U^{-1}}(x)$ on the stopping region $[0.39,\bar{x}]$ { (the $\hat b$ in the verification theorem \ref{verification} is equal to $0.39$ in this numerical example)}.

\begin{figure}[!h]
\begin{center}
\centering \includegraphics[width=12cm,height=8cm]{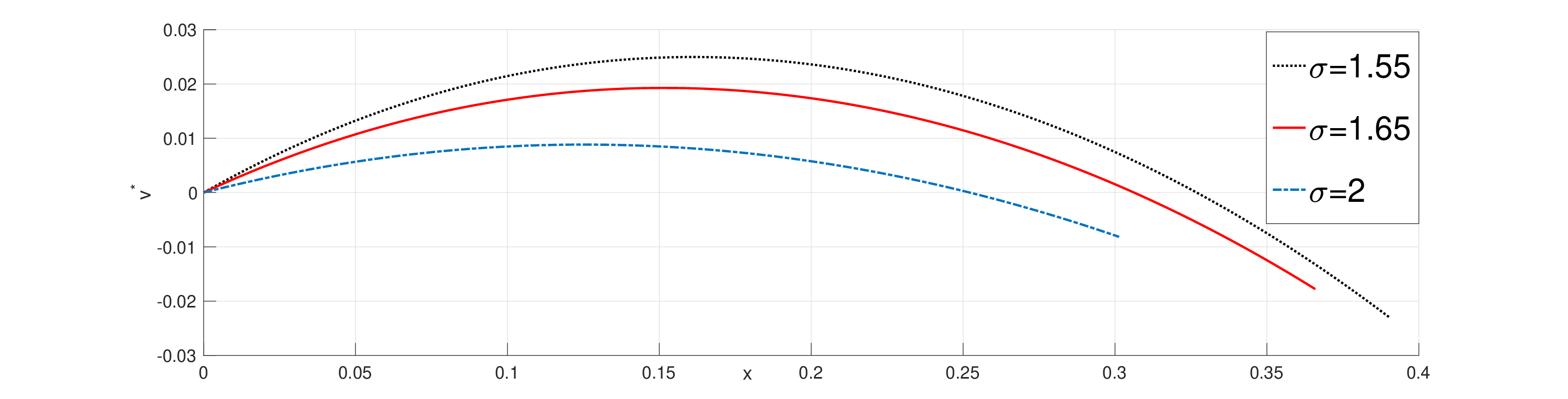}
\caption{ Value function for different $\sigma$.}
\label{FIG1}
\end{center}
	\end{figure}

\noindent 
In Figures \ref{FIG1}--\ref{FIG3},  we focus on the continuation region, and we   represent  (for different values of $\sigma$) respectively the value function, the optimal effort and the optimal rent as function  of the continuation value function of the consortium (denoted by $x$). We provide  some zooms to view some  parts  of the figures  (for $x$ small or  $x$ large) in more details. In particular, we observe that the optimal rent is a decreasing function of the optimal
effort, as in  \cite{sannikov2008continuous}. 

\begin{figure}[!h]
\begin{center}
\centering \includegraphics[width=12cm,height=8cm]{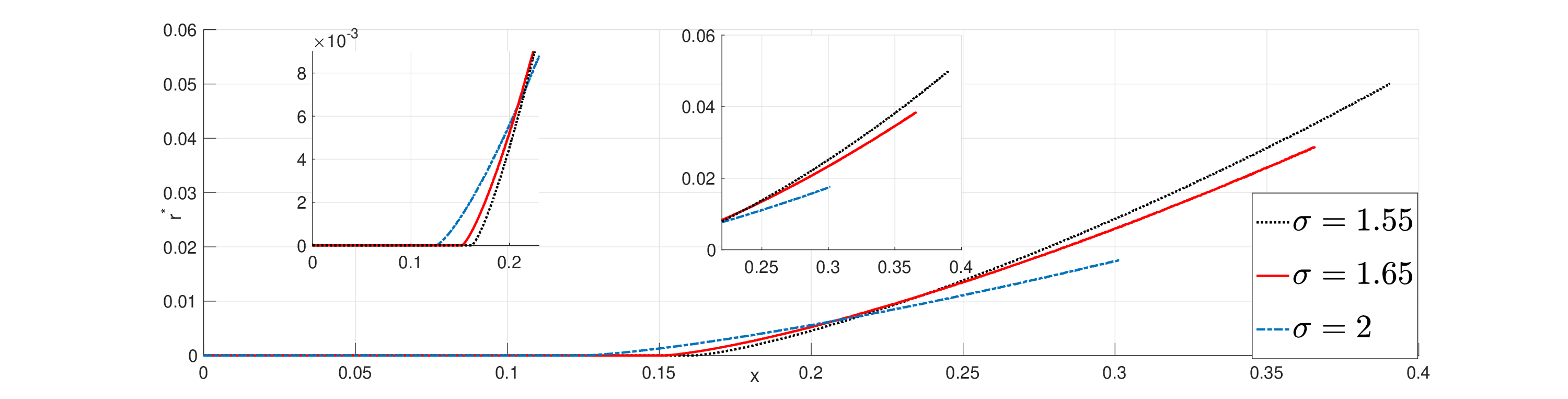}
\caption{ Optimal rent for different $\sigma$.}
\label{FIG2}
\end{center}
	\end{figure}

	\begin{figure}[!h]
\begin{center}
\centering \includegraphics[width=12cm,height=9cm]{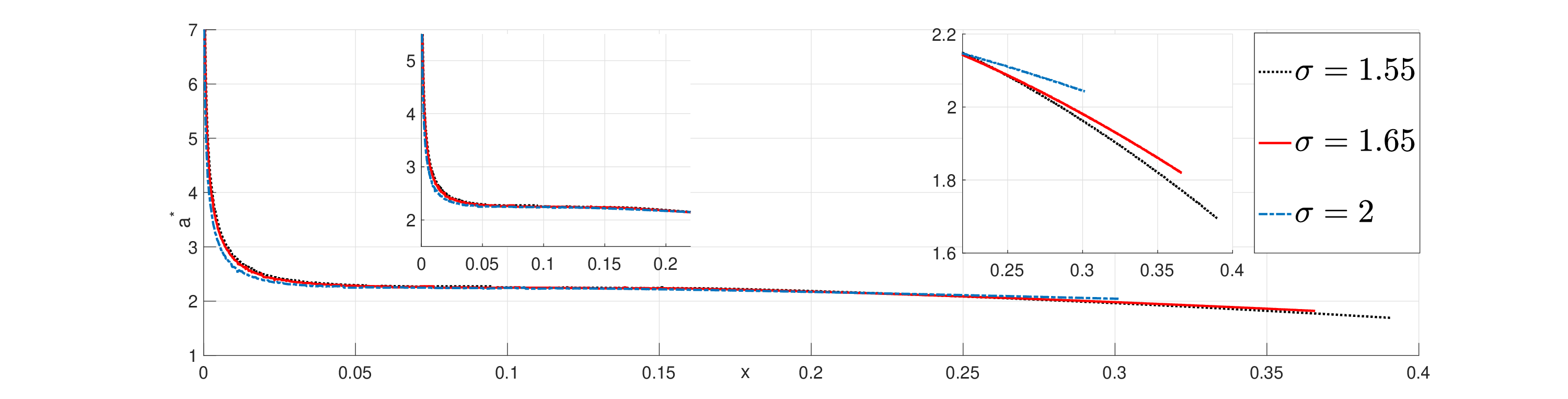}
\caption{ Optimal  effort for different $\sigma.$}
\label{FIG3}
\end{center}
	\end{figure}

\noindent  Figure \ref{fig:sub-first} represents the drift of the continuation value function of the consortium for $\sigma =1.55$. We observe that the maximum is attained for $x_0= 0.16$ which is also the argmax of the public value function (Figure \ref{fig:sub-second}),  as well as the value for which the rent becomes positive (Figure \ref{fig:sub-third}). Indeed,  according  to the verification Theorem  \ref{verification}, the rent is zero when $J^C$ is increasing in $x$.

\begin{figure}[h!]
\begin{subfigure}{0.55\textwidth}
  \includegraphics[width=9.5cm,height=4.5cm]{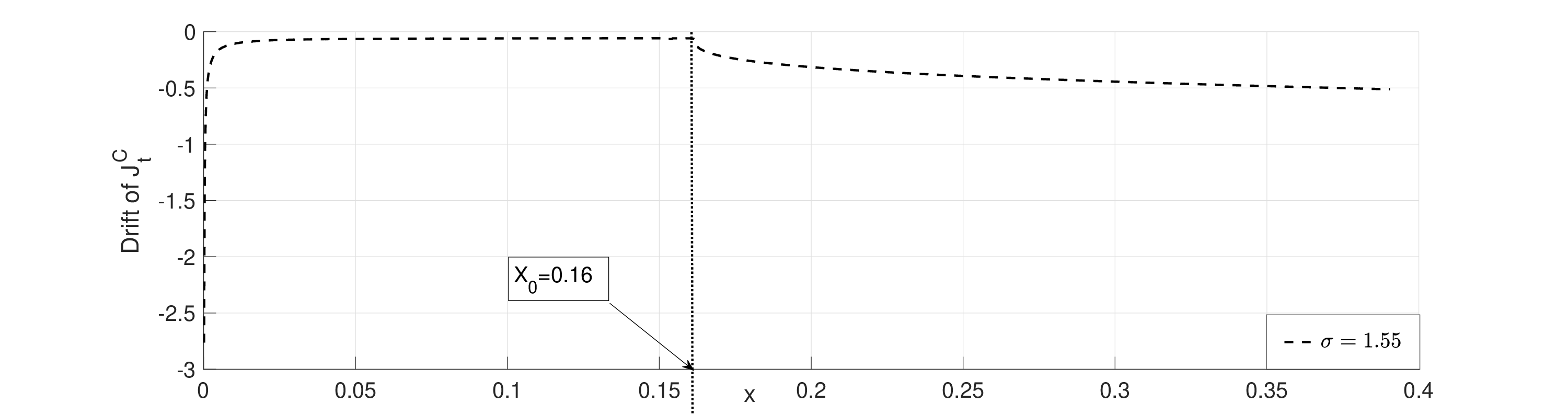}
  \caption{Drift of  the consortium continuation value function.}
  \label{fig:sub-first}
\end{subfigure}
\newline
\begin{subfigure}{0.55\textwidth}
  \includegraphics[width=9.5cm,height=4.5cm]{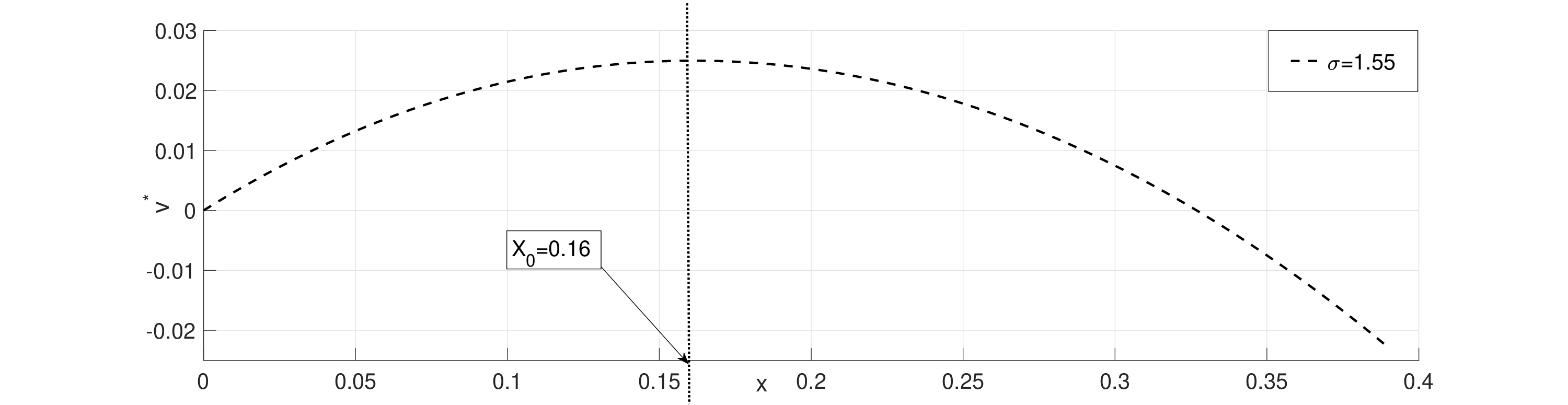}
  \caption{Value function}
  \label{fig:sub-second}
\end{subfigure}\\
\begin{subfigure}{0.55\textwidth}
  \includegraphics[width=9.5cm,height=4.5cm]                {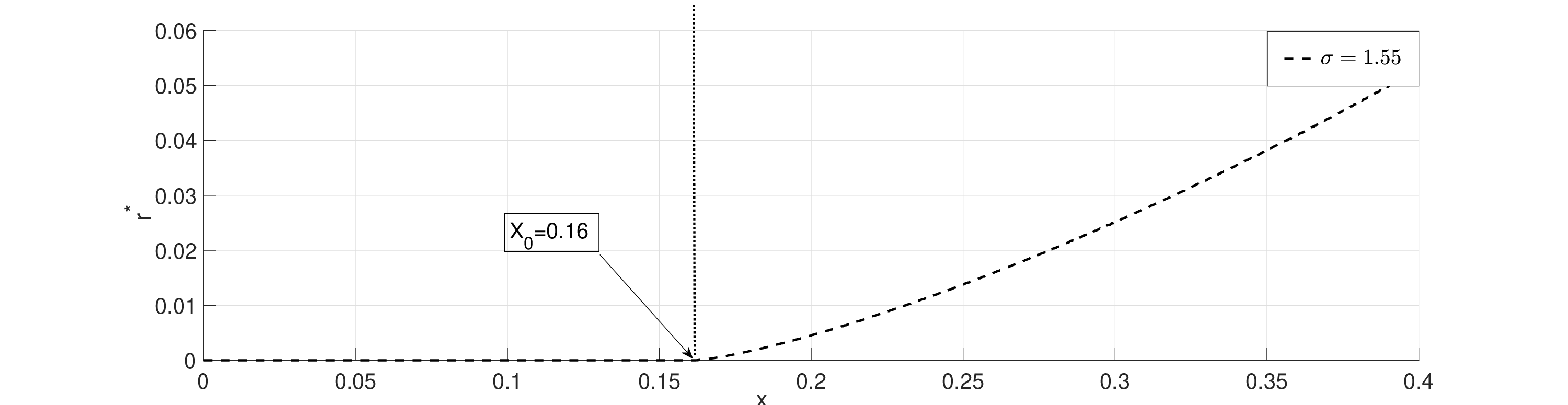}
  \caption{Optimal rent}
  \label{fig:sub-third}
\end{subfigure}
\caption{
}
\end{figure}
\noindent Those  numerical results are in accordance with \cite{sannikov2008continuous}. 
 In addition, we also provide an analysis of the sensitivity with respect to the volatility $\sigma$. 

 \paragraph*{Sensitivity of the results to the  parameter $\sigma$}
The sensitivity to the volatility is  an important features for PPP contracts, since this component, that is a high  noise  of the project's social  welfare (that  is high risk)  is often an argument used by politicians for justifying to  resort on PPP contrats. {  Indeed, in case of high uncertainty,  small/medium public entities (such as  local communities) often prefer to outsource at a large consortium which is (financially) stronger to face the risk.}
 We observe that the optimal public value function $v$ is  decreasing with respect to $\sigma$ (see Figure \ref{FIG1}), which means that more risk implies  less profit.
Besides,  the monotonicity of the optimal controls $a$ and $r$ with respect to the volatility depends  on the level of the consortium's continuation value function $x$ (see Figures  \ref{FIG2} and \ref{FIG3}). 
 When $x$ is small,  the consortium provides more effort as the volatility increases, to compensate the impact of adverse scenarios of the noise $\sigma W$ and to increase the  social welfare. The monotonicity is the opposite  when  $x$ is large: if the consortium's continuation value function is already high, then the impact of a higher effort  may be partly offset in case of adverse scenarios and it is not worth for the consortium to provide more effort if the volatility increases. The inverse monotonicity holds for the rent. \\
Starting from a certain threshold  $x_0$ of  $x$,  the rent becomes positive.  The threshold  $x_0$ corresponds to the  consortium value for which the drift of the  consortium  continuation function $J^C$ is maximum and for which the public value function $v$  is maximum (see Figures  \ref{fig:sub-first}  and \ref{fig:sub-second}). The interval $[0,x_0]$ is called {\it probationary interval} in  \cite{sannikov2008continuous}.
The higher the  volatility, the  lower is the threshold  $x_0$ (see Table \ref{tablex0}).  This means that  more volatility makes the public giving a non-zero rent to encourage the agent to make effort. In the meantime, the  higher the volatility  the smaller the continuation region. 
\begin{table}[!h]
\centering
    \begin{tabular}{|l|c|c|c|c|c|r|}
    \hline
     $\sigma$  & Continuation region  &Threshold $x_0$ \\
    \hline
    $1.55$ & $0.390$ & 0.160\\
     \hline
      $1.65$ & $0.3655$  & 0.156\\
      \hline
           $2$ & $0.3005$  & 0.136\\
      \hline
\end{tabular}
\caption{\textit{The continuation region and the threshold $x_0$ for different  $\sigma$}} 
\label{tablex0}
\end{table}

\newpage

  \paragraph*{ The optimal trajectories} 
Thanks to the verification Theorem \ref{verification}, we  compute the optimal trajectories for  different scenarios. We choose $J_0^C=0.15$,  $\rho=\frac{9}{\sigma^2}$ with $\sigma=1.55$,  a horizon   $T=0.3$ (30 years) 	and  $1500$ time-steps, corresponding to a weekly rebalancing. 
  Figure \ref{FIG5}-\ref{FIG8} represent respectively the dynamics of the consortium value  function, the public value function, the optimal rent and the  optimal effort.  
 {  The public stops the contrat when the  consortium value function  $(\hat J_t^C)_t$ reaches the level $\hat b=0.39$ or $0$, which are  solutions of the equation $v(x)+U^{-1}(x)=0$. 
  Indeed, as long as $x \in (0, 0.39)$ (which is the continuation region), $v(x)+U^{-1}(x)>0$. 
  As soon as $x$ reaches the  stopping region  $[0.39,\infty)$, the contract stops and $v(x)+U^{-1}(x)=0$.
 In four scenarios, the contract stops when the consortium value function hits the level  $0.39$: the public stops the contract because it becomes to costly to incentivize the consortium.  In this situation the terminal  payment is  equal to ${   U^{-1}}(0.39)$.   In one scenario (in pink) the contract stops when the consortium value function hits zero,  meaning  the consortium's bankruptcy: the consortium will provide no more effort in the future and the contract stops.}  This trajectory corresponds to an adverse scenario for the noise $\sigma W$ with an accumulation of negative  increments { for $t > 0.094$}.  Despite the increasing  efforts of the consortium,  the social welfare remains decreasing and the public puts the rent to zero. This explain why the consortium value function is rapidly decreasing to zero in this scenario. 
 We also check on 10 000 scenarios that the assumption   $\esp^\pr[e^{ (\rho-2\lambda)\tau^*}\widehat{J_{\tau^*}^C}^{2}] < \infty$ in  the verification Theorem \ref{verification} is satisfied. 
 
 \newpage 
 
\begin{figure}[!h]
\begin{center}
\centering \includegraphics[width=14cm,height=9cm]   {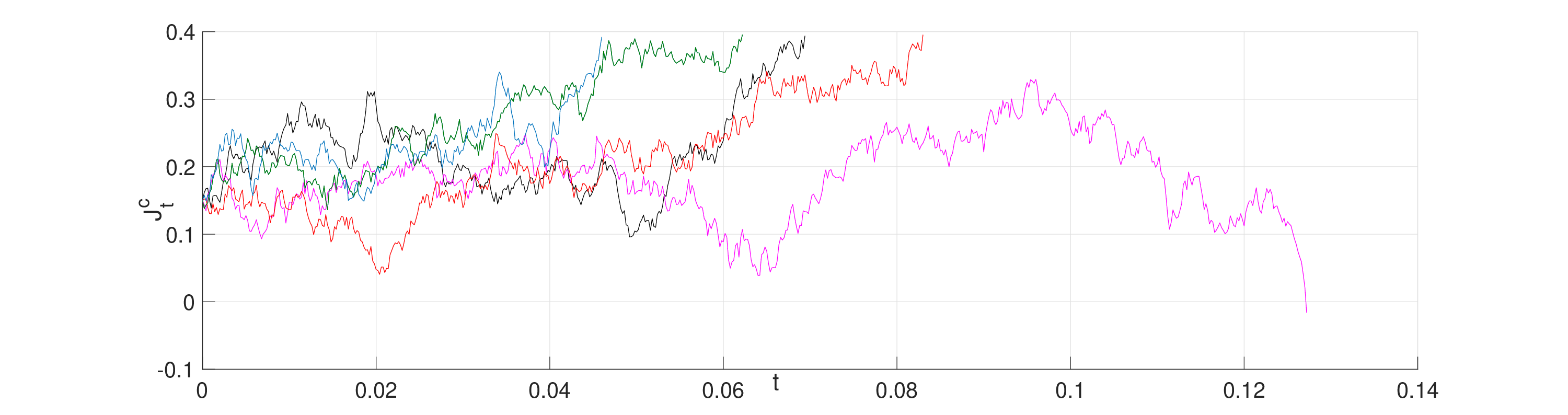}
\caption{The consortium value function}
\label{FIG5}
\end{center}
	\end{figure}
	
	\begin{figure}[!h]

\vspace*{-1cm}

\begin{center}
\centering \includegraphics[width=14cm,height=9cm]   {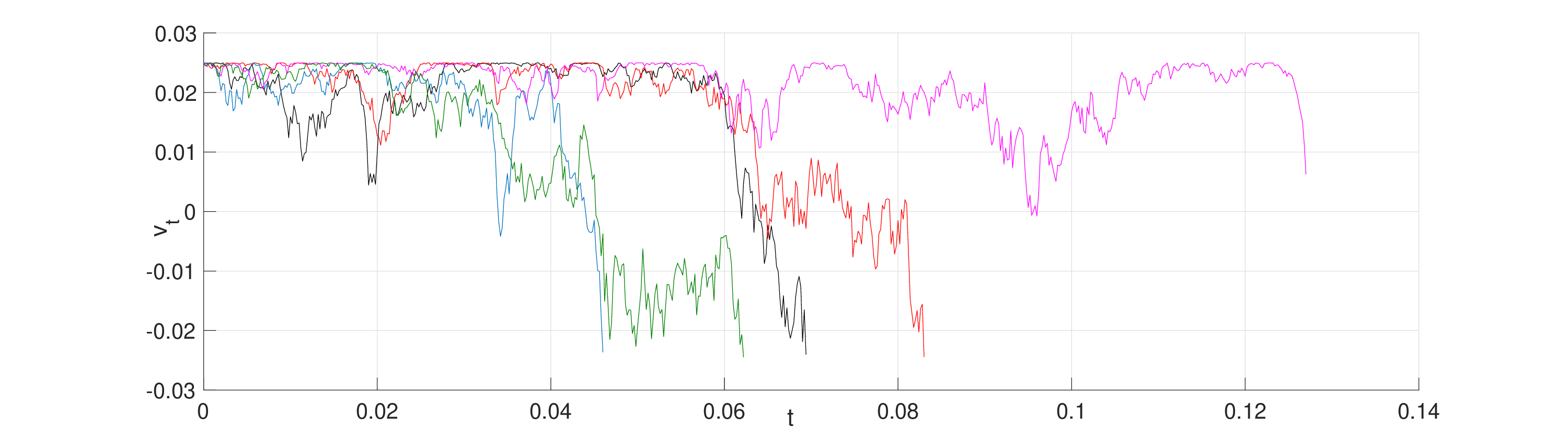}
\caption{ The public value function }
\label{FIG6}
\end{center}
	\end{figure}

\begin{figure}[!h]
\begin{center}
\centering \includegraphics[width=14cm,height=9cm]   {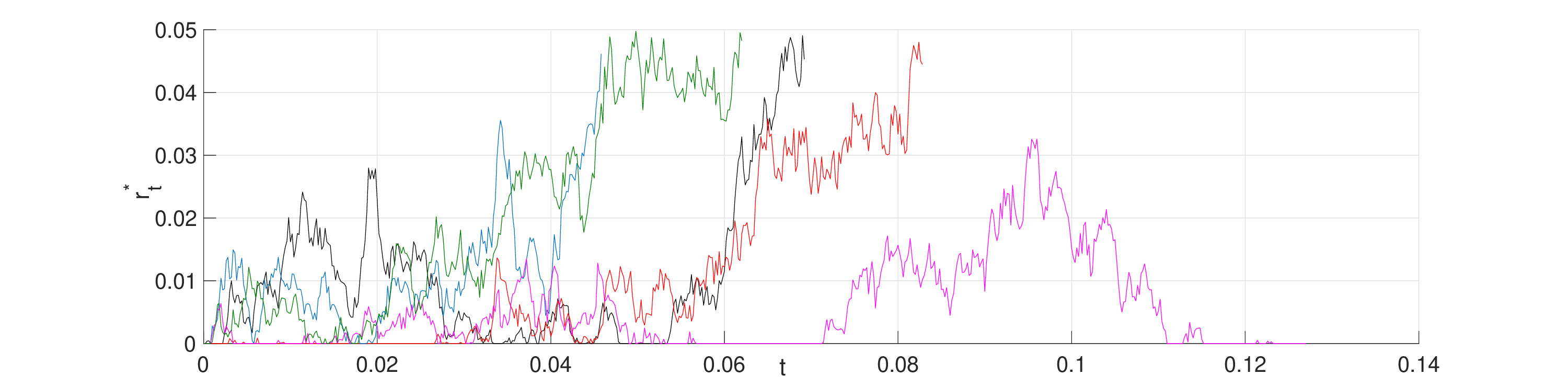}
\caption{ The optimal rent}
\label{FIG7}
\end{center}
	\end{figure}
	
	\vspace*{-1cm}
	
	\begin{figure}[!h]
\begin{center}
\centering \includegraphics[width=14cm,height=9cm]   {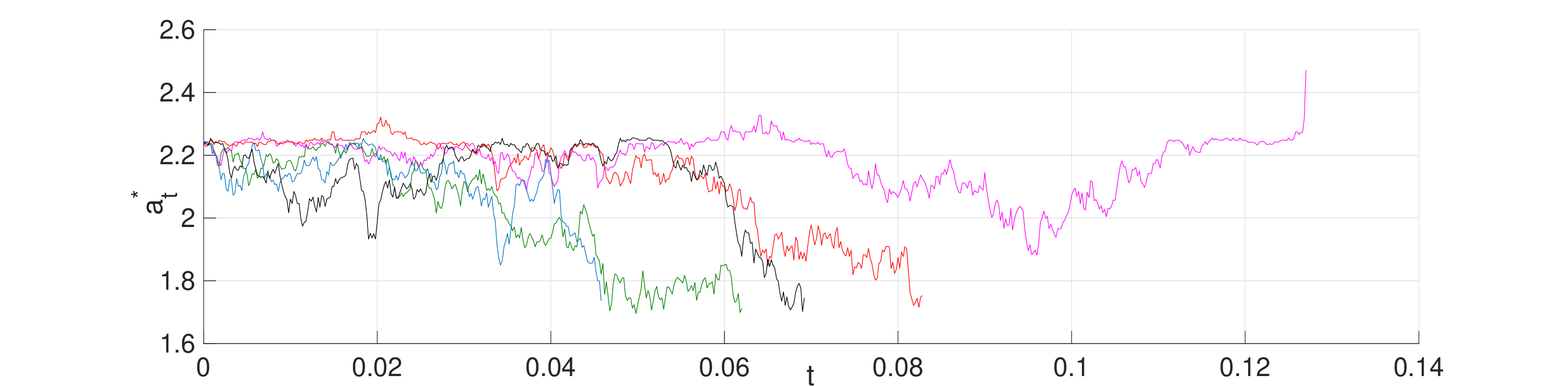}
\caption{The  optimal effort}
\label{FIG8}
\end{center}
	\end{figure}

\newpage
\newpage 

\paragraph{Conclusion}
This paper computes  the optimal Public Private Partnership (PPP) contract between a consortium that provides a non-observable  effort, and a public entity that pays him a continuous rent. When the contract becomes too unfavorable, the public can stop the contract. In this context of moral hazard, we solve this  optimal stochastic control with optimal stopping problem   { by establishing  a one-to-one correspondence between the continuation value of the consortium and the  optimal contract  payments,  using   backward stochastic differential equations  with random terminal time.} A special attention is paid on the explicit characterization of the optimal contract. 
An analysis using viscosity solutions could have weakened the assumptions. Nevertheless, although the
verification Theorem \ref{verification}  is obtained under strong assumptions, it allows us to exhibit the optimal controls and the optimal contract in a feedback form, that are thus numerically implementable. One interesting and important additional consideration   would be to calibrate those numerical results on PPP data (that are difficult to obtain), in order to help public authorities in designing the optimal contract  for the financing and  maintenance of public infrastructures. \\
 { To conclude,  our analysis seems to reveal  that PPP contracts are not very satisfactory  in the long run for the public,   due in  particular to information asymmetry and moral hazard.   
 In fact,  the  main advantage of PPP contracts for the public entity  is to outsource the investment (and thus the debt), see \cite{espinosa2016reducing}. This a  short term advantage that seems not to be sufficient to compensate the long run drawbacks of PPP contracts. 
For example  very recently in France,  the  government   is turning back on the use of such  contracts and is  coming  back to standard commissioning of public works. }
 }

\section{Appendix} 
\subsection{{{\bf Proof of Proposition \ref{EDSR}}}}
We  prove of the existence and uniqueness of a solution to the BSDE (\ref{BSDE02}) whose  generator does not depend in $Y$. Assuming $\rho > C^2_g$,
we  prove directly that  $Y\in\mathcal{S}^2_\rho(\tau)$ and we construct a contraction $\phi$ with respect only to $Z$ from $\mathcal{H}^2_\rho(\tau)$ onto $\mathcal{H}^2_\rho(\tau)$. We recall that in Darling and Pardoux \cite{darling1997backwards} studied a BSDE  with finite random horizon and  the generator of BSDE depends on $(Y,Z)$, they make  intermediate steps, first they showed that $Y\in\mathcal{H}^2_\rho(\tau)$, then they constructed a contraction from $\mathcal{H}^2_\rho(\tau)\times\mathcal{H}^2_\rho(\tau)$ onto $\mathcal{H}^2_\rho(\tau)\times\mathcal{H}^2_\rho(\tau)$.   {We recall that in our case, we cover  the finite and the infinite case.}\\
\noindent We give the proof of the existence and uniqueness of a solution $(Y,Z)$  to the BSDE $(\tau,\xi,g)$ in $\mathcal{S}^2_\rho(\tau)\times\mathcal{H}^2_\rho(\tau)$
$$Y_t=\zeta{\un_{\{\tau<\infty\}}}+\int_t^\tau g(s,\omega,Z_s)ds-\int_t^\tau Z_sdW_s.$$
Let  $v=(V_s)_{s\geq0}\in\mathcal{H}^2_\rho(\tau)$ and we associate
\begin{equation}\label{111}
Y_{t\wedge\tau}=\esp\left[\zeta{\un_{\{\tau<\infty\}}}+\int_{t\wedge\tau}^\tau g(s,\omega,V_s)ds|\mathcal{F}_{t\wedge\tau}\right].
\end{equation}
We consider the martingale $M_{t\wedge\tau}=\esp\left[\zeta{\un_{\{\tau<\infty\}}}+\int_0^\tau g(s,\omega,V_s)ds|\mathcal{F}_{ {t\wedge\tau}}\right]$ which is square integrable under the assumptions on $(\zeta,g).$ By the martingale representation theorem, there exists a unique process $Z\in\mathcal{H}^2_0(\tau)$ such that
$$M_{t\wedge\tau}=M_0+\int_0^{t\wedge\tau} Z_s dW_s.$$
We  have
\begin{eqnarray*}
Y_{t\wedge\tau}&=&\esp\left[\zeta{\un_{\{\tau<\infty\}}}+\int_{t\wedge\tau}^\tau g(s,\omega,V_s)ds|\mathcal{F}_{t\wedge\tau}\right]\\
&=&M_{t\wedge\tau}-\int_0^{t\wedge\tau} g(s,\omega,V_s)ds\\
&=&M_0-\int_0^{t\wedge\tau}g(s,\omega,V_s)ds+\int_0^{t\wedge\tau}Z_sdW_s.
\end{eqnarray*}
Observe by Doob's inequality that
$$\esp\left[\Sup_{0\leq t\leq \tau}|\int_{t\wedge\tau}^\tau Z_sdW_s|^2\right]\leq 4\esp\left[\int_0^\tau|Z_s|^2ds\right]<\infty.$$
Under {  Assumptions $(H1(\rho))$-$(H2)$}, we deduce that $Y$ lies in $\mathcal{S}^2_0(\tau).$ \\
\pf ~\underline{We show that $Z\in\mathcal{H}_\rho^2(\tau).$}\\
 Let $n\in\N$, we denote
$${\tau_n:=\inf\{s\geq 0: e^{\rho s}|Y_s|>n\}\wedge \tau{\wedge n}.}$$
Applying It\^o's formula to the process $(e^{\rho t}|Y_t|^2)_{t\geq0}$ between time $0$ and $\tau_n${ , we have}
$$|Y_0|^2+\int_{0}^{\tau_n}e^{\rho r}|Z_r|^2dr=e^{\rho\tau_n}|Y_{\tau_n}|^2+\int_{0}^{\tau_n}e^{\rho r}(-\rho|Y_r|^2+2Y_rg(r,\omega,V_r))dr-\int_{0}^{\tau_n}2e^{\rho r}Y_rZ_r dW_r.$$
As $g(t,\omega,.)$ is Lipschitz {with Lipschitz coefficient $C_g$}, we have
$$2yg(t,\omega,v)\leq 2|y||g(t,\omega,0)|+2C_g|y||v|.$$
Using {  the Young's inequality} $2ab\leq \varepsilon a^2+\frac{b^2}{\varepsilon}$   {twice} for $\varepsilon=\demi C_g^2$ and $\varepsilon=\demi$, we obtain
\begin{eqnarray}\label{yg}
2yg(t,\omega,v)&\leq& \demi C^2_g|y|^2+2\frac{|g(t,\omega,0)|^2}{C^2_g}+\demi C^2_g|y|^2+2{|v|^2}\nonumber\\
&=&C^2_g|y|^2+2{|v|^2}+2\frac{|g(t,\omega,0)|^2}{C^2_g},
\end{eqnarray}
which implies
\begin{eqnarray}\label{ito}
\int_{0}^{\tau_n}e^{\rho r}|Z_r|^2dr&\leq& e^{\rho\tau_n}|Y_{\tau_n}|^2+\int_{0}^{\tau_n}e^{\rho r}(-\rho+ C^2_g)|Y_r|^2dr+\int_{0}^{\tau_n}e^{\rho r}(\frac{2}{C^2_g}|g(r,\omega,0)|^2+2|V_r|^2)dr\nonumber\\
&-&\int_{0}^{\tau_n}2e^{\rho r}Y_rZ_r dW_r.
\end{eqnarray}
{From the definition of $\tau_n$, we have 
$$|e^{\rho r}Y_rZ_r|\leq n|Z_r|~\forall~r\in\lbr0,\tau_n\rbr~\pr~a.s.$$
\noindent As $Z\in\mathcal{H}^2_0(\tau)$, the stochastic integral $\left(\int_0^{t\wedge \tau_n}e^{\rho r}Y_rZ_rdW_r\right)_{t\geq0}$ is a martingale.\\
Since   $\rho>C^2_g$, we obtain
\begin{eqnarray}\label{ine1}
\esp\left[\int_0^{\tau_n}e^{\rho r}|Z_r|^2dr\right]&\leq&\esp\left[ e^{\rho\tau_n}|Y_{\tau_n}|^2+\int_0^{\tau_n}e^{\rho r}2{|V_r|^2}dr+\frac{2}{C^2_g}\int_0^{\tau_n}e^{\rho r}|g(r,\omega,0)|^2dr\right]\nonumber\\
&\leq&\esp\left[ e^{\rho\tau_n}|Y_{\tau_n}|^2+\int_0^{\tau}e^{\rho r}2{|V_r|^2}dr+\frac{2}{C^2_g}\int_0^{\tau}e^{\rho r}|g(r,\omega,0)|^2dr\right].
\end{eqnarray}
From Equation (\ref{111}), by using the inequality $2ab\leq a^2+b^2$, Jensen's inequality and since $\tau_n\leq \tau$ a.s., we have
\begin{eqnarray*}
e^{\rho\tau_n}|Y_{\tau_n}|^2&=& e^{\rho\tau_n}\left(\esp\left[\zeta{\un_{\{\tau<\infty\}}}+\int_{\tau_n}^\tau g(s,\omega,V_s)ds|\mathcal{F}_{\tau_n}\right]\right)^2\\
&\leq&2\left(\left(\esp\left[ e^{\frac{\rho\tau_n}{2}}\zeta{\un_{\{\tau<\infty\}}}|\mathcal{F}_{\tau_n}\right]\right)^2+
\left(\esp\left[\int_{\tau_n}^\tau e^{\frac{\rho\tau_n}{2}} |g(s,\omega,V_s)|ds|\mathcal{F}_{\tau_n}\right]\right)^2\right).\\
\end{eqnarray*}
By taking the expectation and using Cauchy-Schwarz inequality, we obtain
\begin{eqnarray}\label{ine2}
\esp\left[e^{\rho\tau_n}|Y_{\tau_n}|^2\right]&\leq& 2\left(\esp\left[e^{\rho\tau_n}\zeta^2{\un_{\{\tau<\infty\}}}\right]+\esp\left[\left(\int_{\tau_n}^\tau e^{\frac{\rho\tau_n}{2}}|g(s,\omega,V_s)|ds\right)^2\right]\right)\nonumber\\
&=&2\left(\esp\left[e^{\rho\tau_n}\zeta^2{\un_{\{\tau<\infty\}}}\right]+\esp\left[\left(\int_{\tau_n}^\tau e^{\frac{\rho}{2}(\tau_n-s)}e^{\frac{\rho s}{2}}|g(s,\omega,V_s)|ds\right)^2\right]\right)\nonumber\\
&\leq&2\left(\esp\left[e^{\rho\tau_n}\zeta^2{\un_{\{\tau<\infty\}}}\right]+\esp\left[\int_{\tau_n}^\tau e^{{\rho}(\tau_n-s)}ds\int_{\tau_n}^\tau e^{{\rho s}}|g(s,\omega,V_s)|^2ds\right]\right)\nonumber\\
&\leq&2\left(\esp\left[e^{\rho\tau_n}\zeta^2{\un_{\{\tau<\infty\}}}\right]+ \esp\left[\frac{1}{\rho}(1-e^{-\rho(\tau-\tau_n)})\int_0^\tau e^{\rho s}(|g(s,\omega,0)|^2+C^2_g|V_s|^2)ds\right]\right)\nonumber\\
&\leq&2\left(\esp\left[e^{\rho\tau_n}\zeta^2{\un_{\{\tau<\infty\}}}\right]+\frac{1}{\rho}\esp\left[\int_0^\tau e^{\rho s}(|g(s,\omega,0)|^2+C^2_g|V_s|^2)ds\right]\right),
\end{eqnarray}
where  the third inequality is obtained by Assumption $(H2).$\\
By (\ref{ine1}) and (\ref{ine2}) and  monotone convergence theorem, we obtain
\begin{eqnarray*}
\esp\left[\int_0^{\tau}e^{\rho r}|Z_r|^2dr\right]&\leq&2\left(\esp\left[e^{\rho\tau}\zeta^2\un_{\{\tau<\infty\}}\right]+\frac{1}{\rho}\esp\left[\int_0^\tau e^{\rho s}(|g(s,\omega,0)|^2+C^2_g|V_s|^2)ds\right]\right)\\
&+&\esp\left[\int_0^{\tau}e^{\rho r}2{|V_r|^2}dr+\frac{2}{C^2_g}\int_0^{\tau}e^{\rho r}|g(r,\omega,0)|^2dr\right].
\end{eqnarray*}
As Assumption $(H1(\rho))$ holds and $V\in\mathcal{H}^2_\rho(\tau)$, we deduce that $Z\in\mathcal{H}^2_\rho(\tau).$\\
\pf~\underline{  We show that $Y\in\mathcal{S}^2_\rho(\tau).$}\\
Applying It\^o's formula to the process $(e^{\rho t}|Y_t|^2)_{t\geq0}$ between time $t\wedge\tau_n$ and $\tau_n$
$$e^{\rho (t\wedge\tau_n)}|Y_{t\wedge\tau_n}|^2+\int_{t\wedge\tau_n}^{\tau_n}e^{\rho r}|Z_r|^2dr=e^{\rho\tau_n}|Y_{\tau_n}|^2+\int_{t\wedge\tau_n}^{\tau_n}e^{\rho r}(-\rho|Y_s|^2+2Y_rg(r,\omega,V_r))dr-\int_{t\wedge\tau_n}^{\tau_n}2e^{\rho r}Y_rZ_r dW_r.$$
Repeating the same argument as in (\ref{yg}) and by   taking $\rho>C^2_g$, we obtain
\begin{eqnarray*}
e^{\rho (t\wedge\tau_n)}|Y_{t\wedge\tau_n}|^2&+&\int_{t\wedge\tau_n}^{\tau_n}e^{\rho r}|Z_r|^2dr\leq e^{\rho\tau_n}|Y_{\tau_n}|^2+\int_{t\wedge\tau_n}^{\tau_n}e^{\rho r}(\frac{2}{C^2_g}|g(r,\omega,0)|^2+2|V_r|^2)dr\\
&-&2\int_{t\wedge\tau_n}^{\tau_n}e^{\rho r}Y_rZ_r dW_r.
\end{eqnarray*}
By Burkholder-Davis-Gundy inequality, we have
\begin{eqnarray*}
\esp\left[\Sup_{0\leq t\leq\tau_n}e^{\rho t}|Y_t|^2\right]&\leq&\esp\left[e^{\rho\tau_n}|Y_{\tau_n}|^2+\int_0^{\tau_n}e^{\rho r}(\frac{2}{C^2_g}|g(r,\omega,0)|^2+2|V_r|^2)dr\right]\\&+&C\esp\left[\left(\int_0^{\tau_n}e^{2\rho r}|Y_r|^2|Z_r|^2dr\right)^\demi\right].
\end{eqnarray*}
On the other hand,
\begin{eqnarray*}
\\C\esp\left[\left(\int_0^{\tau_n}e^{2\rho r}|Y_r|^2|Z_r|^2dr\right)^\demi\right]&\leq& C\esp\left[\Sup_{0\leq t\leq\tau_n}e^{\rho t/2}|Y_t|\left(\int_0^{\tau_n}e^{\rho r}|Z_r|^2dr\right)^\demi\right]\\
&\leq&\demi\esp\left[\Sup_{0\leq t\leq\tau_n}e^{\rho t}|Y_t|^2\right]+\frac{C^2}{2}\esp\left[\int_0^{\tau_n}e^{\rho r}|Z_r|^2dr\right].
\end{eqnarray*}
So, we have
\begin{eqnarray*}
\esp\left[\Sup_{0\leq t\leq\tau_n}e^{\rho t}|Y_t|^2\right]&\leq&\esp\left[e^{\rho\tau_n}|Y_{\tau_n}|^2+\int_0^{\tau_n}e^{\rho r}(\frac{2}{C^2_g}|g(r,\omega,0)|^2+2|V_r|^2)dr\right]\\&+&\demi\esp\left[\Sup_{0\leq t\leq\tau_n}e^{\rho t}|Y_t|^2\right]+\frac{C^2}{2}\esp\left[\int_0^{\tau_n}e^{\rho r}|Z_r|^2dr\right],
\end{eqnarray*}
 we obtain
 \begin{eqnarray*}
 \esp\left[\Sup_{0\leq t\leq\tau_n}e^{\rho t}|Y_t|^2\right]&\leq&2\esp\left[e^{\rho\tau_n}|Y_{\tau_n}|^2+\int_0^{\tau_n}e^{\rho r}(\frac{2}{C^2_g}|g(r,\omega,0)|^2+2|V_r|^2)dr\right]+C^2\esp\left[\int_0^{\tau_n}e^{\rho r}|Z_r|^2dr\right].
 \end{eqnarray*}
{By }using the inequality (\ref{ine2}) and { the} monotone convergence theorem, we obtain
  \begin{eqnarray*}
   \esp\left[\Sup_{0\leq t\leq\tau}e^{\rho t}|Y_t|^2\right]&\leq&2\left(2\bigg(\esp\left[e^{\rho\tau}\zeta^2{\un_{\{\tau<\infty\}}}\right]+{\frac{1}{\rho}}\esp\left[\int_0^\tau e^{\rho s}(|g(s,\omega,0)|^2+C^2_g|V_s|^2)ds\right]\right)\\
 &+& \esp\left[\int_0^{\tau}e^{\rho r}(\frac{2}{C^2_g}|g(r,\omega,0)|^2+2|V_r|^2)dr\right]+C^2\esp\left[\int_0^{\tau}e^{\rho r}|Z_r|^2dr\right]\bigg).
 \end{eqnarray*}
As $(H1(\rho))$ holds, $V\in\mathcal{H}^2_\rho(\tau)$ and  $Z\in\mathcal{H}^2_\rho(\tau),$ we deduce that $Y\in\mathcal{S}^2_\rho(\tau).$\\
\pf~\underline{ We contruct a contraction mapping}
\begin{eqnarray*}
\mathcal{H}^2_\rho(\tau)&\rightarrow&\mathcal{H}^2_\rho(\tau)\\
 V &\mapsto&\Phi(V)=Z
\end{eqnarray*}
defined by
 \begin{equation}
 Y_{t\wedge\tau}=\zeta{\un_{\{\tau<\infty\}}}+\int_{t\wedge\tau}^\tau g(s,\omega,V_s)ds-\int_{t\wedge\tau}^\tau Z_s dW_s.
 \end{equation}
 We denote by
\begin{eqnarray*}
v&:=&V-V',~{z}:=Z-Z',~y:=Y-Y^{'},
\end{eqnarray*}
where $Y$ and $Y'$ are defined from (\ref{111}). {We have $y_\tau=0$, }and
$$dy_t=-\{g(t,\omega,V_t)-g(t,\omega,V'_t)\}dt+z_tdW_t.$$
Applying It\^o's formula to $e^{\rho t}|y_t|^2$ between $0$ and $\tau$, and using { $|g(r,\omega,V_r)-g(r,\omega,V'_r)|\leq C_g| v_r|,$} we obtain
\begin{eqnarray*}
\int_0^\tau e^{\rho r}|z_r|^2dr
 &\leq&\int_0^\tau e^{\rho r}[-\rho  |y_r|^2+2C_g|y_rv_r|]dr-2\int_0^\tau e^{\rho r}y_rz_rdW_r.
\end{eqnarray*}
Taking the expectation and by using {  the Young's inequality} $2ab\leq \varepsilon a^2+\frac{b^2}{\varepsilon}$, {we obtain}
\begin{eqnarray*}
\esp\left[\int_0^\tau e^{\rho r}|z_r|^2dr\right]\leq\esp\left[\int_0^\tau e^{\rho r}[-\rho+\varepsilon C_g^2]|y_r|^2dr\right]+\frac{1}{\varepsilon}\esp\left[\int_0^\tau e^{\rho s}|v_r|^2dr\right].
\end{eqnarray*}
Choosing  $\rho>\varepsilon C^2_g$ and $\varepsilon>1$, it {  yields} that $\Phi$ is a contraction mapping, so there exists a unique solution $({Y},{Z})\in\mathcal{S}^2_\rho(\tau)\times\mathcal{H}^2_\rho(\tau)$ solving the BSDE (\ref{BSDE02}).
\pf

\subsection{Comparison Theorem}
{\begin{theoreme}\label{comparison}(Comparison Theorem)\\
{Let $(R^i,\tau,\xi^i)\in\mathcal{A}^P_{  \rho-2\lambda}$} and  $A^i\in \mathcal{A}^C_{  \rho-2\lambda}$ for $i=1,2,$  and for $\tau\in\mathcal{T}.$  Let $(\tilde{Y}^i,\tilde{Z}^i)\in \mathcal{S}^2_\rho(\tau)\times\mathcal{H}^2_\rho(\tau)$ be the solution of the following BSDE 
\begin{eqnarray}\label{BSDE22}
 \left\{
    \begin{array}{ll}
       d\tilde{Y}_t^i &=-\left(\tilde{U}(R_t^i)+\tilde{\psi}(A_t^i,\tilde{Z}_t^i)\right)dt+\tilde{Z}^i_tdW_t, \\
        \tilde{Y}_\tau^i &=\tilde U(\xi^i){\un_{\{\tau<\infty\}}}.
    \end{array}
\right. 
\end{eqnarray}
If $\xi^1\leq \xi^2~\pr~ a.s$ and 
{\begin{equation}\label{comp1}
\tilde{U}(R_t^1)+\tilde{\psi}(A_t^1,\tilde{Z}_t^1)\leq \tilde{U}(R_t^2)+\tilde{\psi}(A_t^2,\tilde{Z}_t^1)~\mbox{for all}~t\in\lbr 0,\tau \lbr,~\mathbb{P}a.s.,
\end{equation}}
then
$$\tilde{Y}_t^1\leq \tilde{Y}_t^2~\mbox{for all}~t\in\lbr 0,\tau\lbr~\pr~a.s.$$
\end{theoreme}}
\begin{proof}
Let $n\in\N$, we denote
{$${\tau_n:=\inf\{s\geq 0:|\tilde{Z}_s^1-\tilde{Z}_s^2|>n\}\wedge \tau\wedge n.}$$}
From (\ref{BSDE22}), we have
\begin{eqnarray*}
\tilde{Y}^1_t-\tilde{Y}^2_t&=&\int_t^{\tau_n}\left(\tilde{U}(R_s^1)-\tilde{U}(R_s^2)+\tilde{\psi}(A^1_s,\tilde{Z}_s^{1})-\tilde{\psi}(A^2_s,\tilde{Z}_s^{2})\right)ds-\int_t^{\tau_n}\left(\tilde{Z}_s^1-\tilde{Z}_s^2\right)dW_s\\
&+&\tilde{Y}_{\tau_n}^1-\tilde{Y}_{\tau_n}^2\\
&=&\int_t^{\tau_n}\left(\tilde{U}(R_s^1)-\tilde{U}(R_s^2)+\tilde{\psi}(A^1_s,\tilde{Z}_s^{1})-\tilde{\psi}(A^2_s,\tilde{Z}_s^{2})+\tilde{\psi}(A^2_s,\tilde{Z}_s^{1})-\tilde{\psi}(A^2_s,\tilde{Z}_s^{1})\right)ds\\
&-&\int_t^{\tau_n}\left(\tilde{Z}_s^1-\tilde{Z}_s^2\right)dW_s
+\tilde{Y}_{\tau_n}^1-\tilde{Y}_{\tau_n}^2\\
&\leq& \int_t^{\tau_n}\left(\tilde{\psi}(A_s^2,\tilde{Z}_s^1)-\tilde{\psi}(A_s^2,\tilde{Z}_s^2)\right)ds
-\int_t^{\tau_n}\left(\tilde{Z}_s^1-\tilde{Z}_s^2\right)dW_s\\
&+&\tilde{Y}_{\tau_n}^1-\tilde{Y}_{\tau_n}^2,
\end{eqnarray*}
where the last inequality is obtained by using Inequality (\ref{comp1}).
 We obtain
\begin{eqnarray*}
\tilde{Y}^1_t-\tilde{Y}^2_t&\leq& \int_t^{\tau_n}(\tilde{Z}_s^1-\tilde{Z}_s^2)\frac{\varphi(A_s^2)}{\sigma}ds
-\int_t^{\tau_n}\left(\tilde{Z}_s^1-\tilde{Z}_s^2\right)dW_s
+\tilde{Y}_{\tau_n}^1-\tilde{Y}_{\tau_n}^2\\
&=&-\int_t^{\tau_n}\left(\tilde{Z}_s^1-\tilde{Z}_s^2\right)dW_s^{A^2}
+\tilde{Y}_{\tau_n}^1-\tilde{Y}_{\tau_n}^2,
\end{eqnarray*}
{where $\left(W_t^{A^2}= W_t-\int_0^t\dfrac{\varphi(A_s^2)}{\sigma}ds\right)_{t\geq 0}$ is $\pr^{A^2}$-Brownian motion.}\\
By taking the {conditional} expectation under $\pr^{A^2}$, the stochastic integral $\int_t^{\tau_n}(\tilde{Z}_s^1-\tilde{Z}_s^2)dW_s^{A^2}$ vanishes. {As $\tilde{Y}^i$ is $\mathbb{F}$-progressively measurable process,}   we obtain $$ \tilde{Y}^1_t-\tilde{Y}^2_t\leq \esp^{A^2}\left[\tilde{Y}_{\tau_n}^1-\tilde{Y}_{\tau_n}^2|\mathcal{F}_t\right].$$
 Since $\hat{p}\in(2,\infty),$ there exists a unique $\varepsilon\in(0,1)$ such that $\hat{p}=2\frac{1+\varepsilon}{1-\varepsilon}$. We define $p:=1+\varepsilon \in (1,2).$ We fix $q_1:=\frac{2}{2-p}$ and let $q_2$ be the conjugate of $q_1$ i.e. $q_2=\frac{2}{p}$. 
\begin{eqnarray*}
 \Sup_{n\in\N}\esp^{A^2}\left[|\tilde{Y}_{\tau_n}^1-\tilde{Y}_{\tau_n}^2|^p|\mathcal{F}_t\right]&=&\Sup_{n\in\N}\esp\left[\Big|\frac{\gamma^{ A^2}_{\tau_n}}{\gamma_t^{A^2}}(\tilde{Y}_{\tau_n}^1-\tilde{Y}_{\tau_n}^2)\Big|^p|\mathcal{F}_t\right]\\ 
&\leq&C\esssup_{\tau\in\mathcal{T}_t}\esp\left[\Big|\frac{\gamma_\tau^{ A^2}}{\gamma_t^{A^2}}\Big|^{\hat{p}}|\mathcal{F}_t\right]^{\frac{2-p}{2}}
 \Bigg(\esp\left[\Sup_{0\leq s\leq \tau}|\tilde{Y}^1_s|^{2}|\mathcal{F}_t\right]
+\esp\left[\Sup_{0\leq s\leq \tau}|\tilde{Y}^2_s|^{2}|\mathcal{F}_t\right]\Bigg)^{\frac{p}{2}},
\end{eqnarray*}
where the first equality is obtained using Bayes formula.\\
{ As $\Sup_{0\leq s\leq \tau}|\tilde{Y}_s^i|\in L^2(\pr)~\forall i\in\{1,2\}$, the conditional expectation $\esp\left[\Sup_{0\leq s\leq \tau}|\tilde{Y}_s^i||\mathcal{F}_t\right]\in L^2(\pr)$. \\In addition $A\in\mathcal{A}^{\hat{p}}$ which implies $ \Sup_{n\in\N}\esp^{}\left[\Big|\frac{\gamma_{\tau_n}^{A^2}}{\gamma_{t}^{A^2}}(\tilde{Y}_{\tau_n}^1-\tilde{Y}_{\tau_n}^2)\Big|^p|\mathcal{F}_t\right]<\infty,$} therefore $\left(\tilde{Y}_{\tau_n}^1-\tilde{Y}_{\tau_n}^2\right)_{n\in\N}$ is uniformly integrable under $\pr^{A^2}$. This implies the convergence in $L^1(\pr^{A^2})$. We may pass to the limit as $n\rightarrow\infty$,
$$\lim_{n\rightarrow\infty}\esp^{A^2}\left[\tilde{Y}_{\tau_n}^1-\tilde{Y}_{\tau_n}^2|\mathcal{F}_t\right]=\esp^{A^2}\left[\lim_{n\rightarrow\infty}(\tilde{Y}_{\tau_n}^1-\tilde{Y}_{\tau_n}^2)|\mathcal{F}_t\right]=\esp^{A^2}\left[\tilde{Y}_{\tau}^1-\tilde{Y}_{\tau}^2|\mathcal{F}_t\right]\leq0.$$
Therefore, {$\tilde{Y}^1_t\leq \tilde{Y}^2_t, \pr ~a.s,\forall t\in\lbr 0,\tau \lbr.$ By continuity of $\tilde{Y}^1$ and $\tilde{Y}^2,$ we have $$\tilde{Y}_t^1\leq \tilde{Y}^2_t,~\forall t\in\lbr 0,\tau \lbr~\pr~a.s.$$}\pf
\end{proof}

\end{document}